\documentclass[11pt,english]{article}
\usepackage[margin=1in]{geometry}
\usepackage[utf8]{inputenc}
\usepackage[T1]{fontenc} 
\usepackage{amsmath,amssymb}
\usepackage{amsthm}
\usepackage{authblk}
\usepackage{mathtools}
\usepackage{hyperref}
\usepackage{cleveref}
\usepackage[inline]{enumitem}
\usepackage{fourier}
\usepackage[textsize=small]{todonotes}
\usepackage{graphicx} 
\usepackage[linesnumbered,commentsnumbered,ruled,vlined]{algorithm2e}

\usepackage{cleveref}
\usepackage[inline]{enumitem}
\usepackage{fourier}
\usepackage[textsize=small]{todonotes}
\setuptodonotes{inline}

\usepackage{thmtools}
\usepackage{thm-restate}

\usepackage[many]{tcolorbox}
\newtcolorbox[
    auto counter,
    number within=section,
    crefname={subroutine}{Subroutine}]%
  {subroutineBox}[2][]{
    colback=white,
    code={\def\mytitle{#2}},
    title={\bf Oracle \thetcbcounter.} \mytitle,%
    sharp corners,
    #1}

\setlist[enumerate]{label=(\roman*)}

\newtheorem{theorem}{Theorem}
\newtheorem{proposition}[theorem]{Proposition}
\newtheorem{lemma}[theorem]{Lemma}
\newtheorem{corollary}[theorem]{Corollary}
\newtheorem{definition}[theorem]{Definition}

\newcommand{\R}{\mathbb{R}}
\newcommand{\Z}{\mathbb{Z}}

\newcommand{\zero}{\mathbf{0}}
\newcommand{\one}{\mathbf{1}}

\newcommand{\slc}{\mathrm{SLC}}
\newcommand{\rk}{\mathrm{rk}}
\newcommand{\im}{\mathrm{im}}

\newcommand{\aug}{\mathrm{aug}}
\newcommand{\supp}{\mathrm{supp}}

\newcommand{\elementary}[1]{\mathcal{E}(#1)}
\newcommand{\mcp}{\mathfrak{m}}
\newcommand{\poly}{\mathrm{poly}}
\newcommand{\pr}[2]{\left\langle #1, #2 \right\rangle}
\newcommand{\mx}{x^{\mathfrak{m}}}
\newcommand{\cx}{x^{\mathrm{cp}}}
\newcommand{\ms}{s^{\mathfrak{m}}}

\DeclareMathOperator*{\argmin}{arg\,min}

\title{On Circuit Diameter and Straight Line Complexity}
\author[1]{Daniel Dadush}
\affil[1]{Centrum Wiskunde \& Informatica, The Netherlands}
\author[2]{Stefan Kober}
\affil[2]{Université Libre de Bruxelles, Belgium}
\author[3]{Zhuan Khye Koh}
\affil[3]{Boston University, USA}

\date{\today}

\begin{document}

\maketitle

\begin{abstract}
The \emph{circuit diameter} of a polyhedron is the maximum length (number of steps) of a shortest circuit walk between any two vertices of the polyhedron.
Introduced by Borgwardt, Finhold and Hemmecke (SIDMA 2015), it is a relaxation of the \emph{combinatorial diameter} of a polyhedron.
These two notions of diameter lower bound the number of iterations taken by circuit augmentation algorithms and the simplex method respectively for solving linear programs.

Recently, an analogous lower bound for path-following interior point methods was introduced by Allamigeon, Dadush, Loho, Natura and Végh (SICOMP 2025).
Termed \emph{straight line complexity}, it refers to the minimum number of pieces of any piecewise linear curve that traverses a specified neighborhood of the central path.

In this paper, we study the relationship between circuit diameter and straight line complexity.
For a polyhedron $P:=\{x\in \R^n: Ax = b, x\geq \zero\}$, we show that its circuit diameter is up to a $\poly(n)$ factor upper bounded by the straight line complexity of linear programs defined over $P$.
This yields a strongly polynomial circuit diameter bound for polyhedra with at most 2 variables per inequality.
We also give a circuit augmentation algorithm with matching iteration complexity.

\end{abstract}
\section{Introduction}
The \emph{simplex} method is probably the most well-known method for solving linear programs (LPs).
Starting from an initial vertex, it moves from vertex to vertex along edges of the polyhedron towards an optimal vertex.
This method is an example of \emph{circuit augmentation} algorithms -- a general algorithmic scheme for solving LPs. 
Like the simplex method, given an initial feasible point, a circuit augmentation algorithm proceeds through a sequence of points on the boundary of the polyhedron. 
However, instead of only moving along edges, the algorithm is allowed to travel along any direction parallel to a maximal subset of facets. 
So, its trajectory may pass through the interior of the polyhedron.

\emph{Interior point methods (IPMs)} are another family of LP algorithms that move in the interior of the polyhedron.
Unlike circuit augmentation algorithms, they always stay far away from the boundary.
In this paper, we compare these two algorithmic frameworks via natural lower bounds that govern their performance.
Ultimately, we are interested in the following question:

\begin{center}
\fbox{
\begin{minipage}{0.92\textwidth}
\centering
\parbox{0.92\textwidth}{
\centering
\smallskip
\itshape Can an algorithm that walks from boundary point to boundary point perform as well as an algorithm that moves in the `deep' interior of the polyhedron?
\smallskip}
\end{minipage}
}
\end{center}
More formally, consider an LP in standard equality form
\begin{equation}\label{lp}
    \min\{\pr{c}{x}:Ax = b, x\geq \zero\},
\end{equation}
where $A\in \R^{m\times n}$, $\rk(A) =m$ and $b\in \R^m$.
Let $P$ denote the feasible region of \eqref{lp}.
For $W:=\mathrm{ker}(A)\subseteq\R^n$, we call $h\in W\setminus\{\zero\}$ an \emph{elementary vector} if it is support-minimal, i.e., there is no $h'\in W\setminus\{\zero\}$ such that $\mathrm{supp}(h')\subsetneq\mathrm{supp}(h)$.
The support of an elementary vector $\supp(h)\subseteq [n]$ is called a \emph{circuit}.
Let $\mathcal E(A)\subseteq W$ and $\mathcal C(A)\subseteq 2^n$ denote the set of elementary vectors and circuits of $A$ respectively.

In a circuit augmentation algorithm, the allowable movement directions are precisely $\mathcal{E}(A)$.
In every iteration $t\geq 0$, the algorithm computes an elementary vector $h\in \mathcal E(A)$ based on a \emph{pivot rule}, and updates the current solution $x^{(t)}$ to $x^{(t+1)} := x^{(t)} + \alpha h$, where $\alpha > 0$ is the maximal step size that preserves feasibility.
The elementary vector $h$ is called an \emph{augmenting direction} if $\pr{c}{h}<0$ and $\alpha>0$.
If no augmenting direction exists, then $x^{(t)}$ is optimal and the algorithm terminates.
The sequence of generated points $x^{(0)}, x^{(1)}, \dots$ is called a \emph{circuit walk}.

Clearly, the number of iterations taken by \emph{any} circuit augmentation algorithm is lower bounded by the maximum length (number of steps) of a shortest circuit walk between any two vertices of $P$.
Introduced by Borgwardt, Finhold and Hemmecke~\cite{BFH15}, this latter quantity is known as \emph{circuit diameter}.
It is a relaxation of the standard \emph{combinatorial diameter} --  the diameter of the vertex-edge graph of $P$.
Analogously, the combinatorial diameter lower bounds the number of iterations taken by the simplex method.
Currently, the best upper bound on the combinatorial diameter is quasipolynomial~\cite{KK92,T14,S19}, having the form $m^{\log O((n-m)/\log (n-m))}$.

\paragraph{Interior Point Methods}

While circuit augmentation algorithms iterate through points on the boundary of $P$, interior point methods (IPMs) solve \eqref{lp} by staying in the `deep' interior of $P$.
Path-following IPMs reach an optimal solution by following a smooth curve called the \emph{central path}.
Assuming $P$ is bounded and has a strictly feasible solution, this is the parametric curve $\cx:\R_{>0} \to \R^n_{>0}$ defined by
\begin{equation}\label{eq:cp}
    \cx(\mu) := \argmin_{x\in P} \left\{\pr{c}{x} - \mu\sum_{i=1}^n \log (x_i)\right\}.
\end{equation}
The second term in \eqref{eq:cp} is called the \emph{logarithmic barrier}, which pushes the central path away from zero on every coordinate.
As $\mu \to 0$, $\cx(\mu)$ converges to an optimal solution of \eqref{lp}.
Path-following IPMs maintain iterates in a certain neighborhood of the central path while geometrically decreasing $\mu$, and hence the optimality gap.
They can be implemented to run in polynomial time.
Standard analyses yield a bound of $O(\sqrt{n}L)$ on the number of iterations, where $L$ is the total encoding length of $(A,b,c)$. 

\paragraph{Straight Line Complexity}

The trajectory of a path-following IPM is a piecewise-linear curve in a neighborhood of the central path, whose number of pieces corresponds to the number of iterations.
Thus, the minimum number of pieces of \emph{any} piecewise-linear curve in the neighborhood provides a lower bound on the number of iterations taken by the IPM.
In a surprising result, Allamigeon, Benchimol, Gaubert and Joswig~\cite{ABGJ18} constructed a parametric family of LPs such that for suitably large parameter values, any piecewise-linear curve in the neighborhood of the central path has exponentially many pieces.
This result was later generalized to arbitrary self-concordant barriers~\cite{AGV22}.

Recently, Allamigeon, Dadush, Loho, Natura and Végh~\cite{ADL25} complemented this lower bound by giving an IPM whose number of iterations matches it up to $\poly(n)$ factors.
The guarantees of the IPM are stated in terms of a combinatorial proxy of the central path, which we elaborate below.

Let $v^*$ denote the optimal value of \eqref{lp}.
For $g\geq 0$, let
\begin{equation}
    P_g:=\{x\in P:\langle c,x\rangle\le v^* + g\}
\end{equation}
be the feasible sublevel set with optimality gap at most $g$.
Assuming $P_g$ is bounded, the \emph{max central path} is the parametric curve $\mx:\R_{\geq 0}\to \R^n_{\geq 0}$ defined by
\begin{equation}
    \mx_i(g) := \max\{x_i:x\in P_g\} \qquad \qquad \forall i\in [n].
\end{equation}
For each $i\in [n]$, $\mx_i$ is a piecewise-linear, concave and non-decreasing function.

For $\eta\in (0,1]$, we say that a function $f:\R_{\geq 0}\to \R^n_{\geq 0}$ lies in the \emph{$\eta$-neighborhood} of $\mx$ if $\eta \mx \leq f\leq \mx$.
For the logarithmic barrier, its central path $\cx$ lies in the $(1/2n)$-neighborhood of $\mx$.
More generally, for any $\nu$-self-concordant barrier, its central path lies in the $\Omega(1/\nu)$-neighborhood of $\mx$~\cite{ADL25}.
For this reason, the max central path is a good proxy of the central path.
Next, we need the following definition.

\begin{definition}
Let $\phi:\R_{\geq 0}\to\R_{\geq 0}$ be a function and $\eta\in(0,1]$.
The \emph{straight line complexity} of $\phi$ with respect to $\eta$, denoted $\slc_\eta(\phi)$, is the minimum number of pieces of a continuous piecewise linear function $\psi:\R_{\geq 0}\to \R_{\geq 0}$ satisfying $\eta \phi \leq \psi \leq \phi$. 
For $g\geq 0$, we write $\slc_\eta(\phi,g)$ to indicate the analogous quantity for approximating $\phi$ in the interval $[0,g]$. 
\end{definition}

Recall that the trajectory of an IPM is a piecewise-linear curve in a neighborhood of $\cx$.
If it lies in the $\eta$-neighborhood of $\mx$, then its number of pieces is at least $\max_{i\in [n]} \slc_\eta(\mx_i)$. 
The negative result of Allamigeon et al.~\cite{ABGJ18} shows that for any $\eta\in (0,1]$, there exists an LP with $n$ variables and $O(n)$ constraints such that $\max_{i\in [n]} \slc_\eta(\mx_i) = \Omega(2^n)$.
On the positive side, Allamigeon et al.~\cite{ADL25} gave an IPM whose number of iterations is bounded by the sum of straight line complexities of every coordinate of the max central path, i.e.,
\[O\left(\inf_{\eta\in (0,1]} \sqrt{n}\log\left(\frac{n}{\eta}\right) \sum_{i=1}^n \slc_\eta(\mx_i)\right).\]
We remark that the neighborhood parameter $\eta$ is not important, as one can show that for $0<\eta<\eta'<1$, $\slc_{\eta'}(\mx_i) = O(\log(1/\eta)/\log(1/\eta')) \slc_\eta(\mx_i)$. 

\subsection{Our Contributions}

Even though circuit diameter and straight line complexity are two geometric quantities that govern the performance of two broad classes of LP algorithms, not much is known about their relationship.
From \cite{KK92,ABGJ18}, we know that the circuit diameter is smaller than the straight line complexity in the worst case (quasipolynomial vs.~exponential).
However, this is unsatisfactory as it does not provide an instance-by-instance comparison between the two quantities.

Our first result shows that for any LP, the length of a shortest circuit walk from any feasible point to the optimal face can be bounded in terms of the straight line complexity.

\begin{theorem}\label{main}
    Given a bounded LP in the form \eqref{lp}, let $F^*$ be its optimal face.
    For any feasible point $x$ with optimality gap $g\geq 0$, the length of a shortest circuit walk from $x$ to $F^*$ is
    \[O\left(n^2 \log(n) \sum_{i\in [n]} \slc_{1/2}(\mx_i,g) \right).\]
\end{theorem}

Our next result is an algorithmic implementation of \Cref{main}.
In particular, we give a circuit augmentation algorithm that solves an LP in the same number of iterations as the bound in \Cref{main}.
In each iteration, the algorithm calls the \textsc{Ratio-Circuit} oracle to obtain an augmenting direction. 
This oracle takes as input a matrix $A\in \R^{m\times n}$, costs $c\in \R^n$ and weights $v,w\in (\R_{\geq 0}\cup \infty)^n$, and returns a basic optimal solution to the system
\begin{equation}\label{eq:ratio-circuit}
    \min \pr{c}{z} \quad \text{ s.t. } \quad Az = \zero,\quad \pr{v}{z^+} + \pr{w}{z^-}\leq 1,
\end{equation}
where $(z^+)_i := \max\{z_i,0\}$ and $(z^-)_i := \max \{-z_i,0\}$.
Here, we use the convention $ab = 0$ if $a = \infty$ and $b = 0$.
If bounded, a basic optimal solution is either $\zero$ or an elementary vector $z\in \mathcal{E}(A)$ that minimizes $\pr{c}{z}/(\langle v, z^+\rangle + \pr{w}{z^-}
)$.

\begin{theorem}\label{main_alg}
There exists a circuit augmentation algorithm such that given any bounded LP in the form \eqref{lp} and any feasible point $x$ with optimality gap $g\geq 0$, it returns an optimal solution using 
\[O\left(n^2 \log(n) \sum_{i\in [n]} \slc_{1/2}(\mx_i,g) \right)\]
\textsc{Ratio-Circuit} augmentation steps.
\end{theorem}

Qualitatively, this signifies that circuit augmentation algorithms can achieve the convergence guarantees of path-following IPMs.
In other words, moving in the `deep' interior of the polyhedron is not more powerful than iterating along boundary points via circuits.

As an immediate corollary, the circuit diameter of a polyhedron is at most the straight line complexity of its associated LPs up to $\poly(n)$ factors.

\begin{corollary}\label{main_diameter}
For a polyhedron $\{x\in \R^n:Ax = b, x\geq \zero\}$, its circuit diameter is
\begin{equation*}
    O\left(n^2\log(n) \sup_{c\in\R^n}\left\{\sum_{i\in[n]}\slc_{1/2}(\mx_i)\right\}\right).
\end{equation*}
\end{corollary}

Consequently, for a polyhedron, if the straight line complexities of its associated LPs are strongly polynomial, so is its circuit diameter.
Recently, strongly polynomial bounds were obtained for the straight line complexity of LPs in the form \eqref{lp} whose constraint matrix $A$ has at most 2 non-zeroes per column~\cite{DKNOV24}.
In particular, $\slc_{1/2}(\mx_i) =  O(mn\log n)$ for all $i\in [n]$.
This class of LPs is equivalent to the minimum-cost generalized flow problem.
Applying \Cref{main_diameter} yields a strongly polynomial circuit diameter bound for these polyhedra.

\begin{corollary}\label{diameter_genflow}
For a polyhedron $\{x\in \R^n:Ax = b, x\geq \zero\}$ where $A\in \R^{m\times n}$ has at most 2 non-zeroes per column, its circuit diameter is $O(mn^4 \log^2n)$.
\end{corollary}

Previously, Kalai~\cite{K92} gave an upper bound of $(n-m)^{\log_{11/10}19} = (n-m)^{30.893\dots}$ on the combinatorial diameter of generalized flow polyhedra with arc capacities, where $m$ and $n$ are the number of nodes and arcs respectively.
With arc capacities, $\slc_{1/2}(\mx_i) = O(n^2\log n)$ for all $i\in [n]$, so our bound in \Cref{diameter_genflow} becomes $ O(n^6 \log^2 n)$.

Further, it was shown in \cite{ADL25} that the straight line complexity of the primal and dual variables are essentially the same.
In particular, for the dual LP of \eqref{lp}
\begin{equation}\label{dlp}
   \max\{\langle b,y \rangle:A^\top y + s = c, s\geq \zero\},
\end{equation}
we have $\slc_\eta(\ms_i) = O(\slc_\eta(\mx_i))$ for all $i\in [n]$.
Hence, we also obtain a strongly polynomial circuit diameter bound for polyhedra with at most 2 variables per inequality.

\begin{corollary}\label{diameter_2vpi}
For a polyhedron $\{y\in \R^m:A^\top y\leq c\}$ where $A^\top \in \R^{n\times m}$ has at most 2 non-zeroes per row, its circuit diameter is $O(mn^4 \log^2n)$.
\end{corollary}

Both of these bounds are an improvement over those in terms of the circuit imbalance measure~\cite{DKNV24}, which could be unbounded for such polyhedra.

\subsection{Technical Overview}

The starting point of proving \Cref{main} is to decompose the max central path $\mx$ into  at most $2\sum_{i\in [n]}\slc_{1/2}(\mx_i)$ \emph{polarized segments}. 
A segment of the max central path $\{\mx(g):g\in [g_0, g_1]\}$ is \emph{polarized} if the variable set $[n]$ can be partitioned into $B\cup N$ such that the coordinates in $B$ are barely changing, while the coordinates in $N$ are scaling down linearly with the optimality gap $g$.
More precisely, for every $g\in [g_0,g_1]$, we have
\begin{equation}\label{eq:polarized}
\begin{aligned}
    \frac14 \mx_i(g_1) &\leq \mx_i(g) \leq \mx_i(g_1) \qquad \quad\;\;\forall i\in B \\
    \frac{g}{g_1} \mx_i(g_1) &\leq \mx_i(g) \leq \frac{4g}{g_1} \mx_i(g_1) \qquad \forall i\in N
\end{aligned}
\end{equation}
We call $[g_0,g_1]$ a \emph{polarized interval}.
Since the max central path multiplicatively approximates the central path, the latter also admits the same polarized decomposition.
The connection between straight line complexity and polarization was already pointed out in the work of \cite{ADL25}. 

Given a current point $x\in P$ with optimality gap $g$, let $[g_0,g_1]$ be the polarized interval containing $g$.
To finish the proof, it suffices to traverse the interval in $\poly(n)$ iterations, i.e., reach a point $x'\in P$ with optimality gap $g'\leq g_0$.
We will take two types of circuit steps, which can be obtained using the \textsc{Ratio-Circuit} oracle.
The first type of circuit step is obtained by setting $v = \zero$ and $w = 1/x$.
First proposed by Wallacher~\cite{W89} for minimum-cost flow, it decreases the optimality gap by a factor of at least $1-1/n$~\cite{MS00}.
So, if $g_1/g_0 \leq 2^{\poly(n)}$, then repeating this step for $\poly(n)$ iterations does the job.
We remark that for flow problems, such a minimum-ratio cycle can be found in strongly polynomial time.

On the other hand, if $g_1/g_0 \gg 2^{\poly(n)}$, then we can no longer rely on the geometric decay of the optimality gap.
Let $(B,N)$ be the bipartition associated with this polarized interval.
From \eqref{eq:polarized}, we know that the contribution of $B$ to the optimality gap in this interval is negligible.
So, the goal is to decrease the coordinates in $N$.

For simplicity of notation, we now assume that $\mx(g_1) = 1_n$ \footnote{This  is achieved by rescaling the instance: $A \rightarrow A {\rm diag}(\mx(g_1))$ and $c \rightarrow {\rm diag}(\mx(g_1)) c$.}. Unlike an IPM, a circuit step inevitably sets a coordinate to zero.
So, one needs to be careful about which coordinates are zeroed out in order to avoid oscillations.
To this end, we consider the \emph{lifting operator} $\ell^W_N:\pi_N(W)\to \pi_B(W)$
\begin{equation}\label{eq:lift}
    \ell^W_N(z) = \argmin_{(y_B,y_N)\in W}\left\{\|y_B\|: y_N = z_N\right\},
\end{equation}
where $\pi_B$ and $\pi_N$ denote the coordinate projections onto $B$ and $N$ respectively.
Using its singular values $\sigma_1\geq \sigma_2\geq \dots \geq \sigma_{\dim(\pi_N(W))}$, we further subdivide $[g_0,g_1]$ into subintervals 
\begin{equation}\label{eq:singval_intervals}
    \left[g_0,\frac{g_1}{\sigma_a}\right], \left[\frac{g_1}{\sigma_a},\frac{g_1}{\sigma_{a+1}}\right], \dots, \left[\frac{g_1}{\sigma_{b}}, g_1\right].
\end{equation}

Now, suppose that $g$ lies in some long subinterval $[g_1/\sigma_{j}, g_1/\sigma_{j+1}]$.
By \eqref{eq:polarized} and \eqref{eq:lift}, any step that significantly decreases the coordinates in $N$ must be close to the singular subspace corresponding to $\sigma_{j+1}, \sigma_{j+2},\dots,\sigma_{\dim(\pi_N(W))}$.
Otherwise, the induced change on $B$ is too big and violates feasibility.
Hence, this suggests to move as much as possible in the singular subspace to ensure significant progress.

Let $S\subseteq N$ be the set of `small' coordinates relative to their max central path values at the end of the subinterval, i.e., $S:=\{i\in[n]:x_i\leq 2^{\poly(n)}\mx_i(g_1/\sigma_{j})\}$.
In the singular subspace, we show the existence of an elementary vector $z\in \mathcal{E}(A)$ such that 1) $\pr{c}{z}\leq -g/(2n)$; 2) $z_S\ge 0$; and 3) $z_S,z_B$ are `small'. 
This motivates our second type of circuit step, where we choose $v$ and $w$ to induce these properties in an optimal solution of \eqref{eq:ratio-circuit}.
This step ensures that a coordinate in $N\setminus S$ is set to zero, while keeping the coordinates in $S$ `small'.
Thus, in $|N|$ iterations, we have $S = N$.
At that point, $g/(g_1/\sigma_{j})\leq 2^{\poly(n)}$, so we can use the first type of circuit step to traverse the subinterval in $\poly(n)$ iterations.

The augmentation scheme described above is almost fully algorithmic, except that it uses the max central path and its decomposition into polarized segments.
We cannot rely on these additional input when proving \Cref{main_alg}.
As our proof of \Cref{main} is robust to $\poly(n)$ approximations, it suffices to be able to approximate them.
A straightforward way of doing so is to simply compute the trajectory of the IPM~\cite{ADL25}.
We show that one can avoid running the IPM by estimating all the required information using $\poly(n)$ calls to \textsc{Ratio-Circuit} together with additional linear algebraic operations.
Details of the algorithmic implementation can be found in \Cref{sec:implementation}.

\subsection{Related Work}

Besides the simplex method, many network optimization algorithms can be seen as circuit augmentation algorithms.
A prominent example is the Edmonds--Karp--Dinic algorithm~\cite{D70,EK72} for maximum flow, which is an efficient implementation of the Ford-Fulkerson method~\cite{FF56}.
Bland generalized this result by giving a circuit augmentation algorithm for LP~\cite{B76,Lee89}.
Building on this work, De Loera, Hemmecke and Lee~\cite{LHL15} analyzed different pivot rules for LP, and also extending them to integer programming.
De Loera, Kafer and Sanità~\cite{LKS22} studied the computational complexity of carrying out these rules, as well as their convergence on 0/1-polytopes.

Dadush, Koh, Natura and Végh~\cite{DKNV24} upper bounded the circuit diameter in terms of the circuit imbalance measure of $A$.
They also gave a circuit augmentation algorithm with a similar iteration bound.
Circuit diameter bounds have been shown for some combinatorial polytopes such as dual transportation polyhedra~\cite{BFH15}, matching, travelling salesman, and fractional stable set polytopes~\cite{KPS19}.
Nöbel and Steiner proved that computing the circuit diameter is strongly NP-hard~\cite{NS25}, which was then strengthened to a $\Pi_2^p$-completeness result by Wulf~\cite{W25}.
Black, Nöbel and Steiner showed that approximating the length of a shortest monotone circuit walk is already hard in dimension 2~\cite{BNS25}.

\subsection{Paper Organization}
We start by providing the necessary preliminaries in \Cref{sec:prelims}.
In \Cref{sec:bound}, we give a procedure which outputs a circuit walk whose length is at most a factor $O(n)$ larger than the bound in \Cref{main}.
Then, we show how to algorithmically implement this procedure in \Cref{sec:implementation}.
Finally, we present an improved amortized analysis in \Cref{sec:amortization}, thereby proving \Cref{main} and \Cref{main_alg}.

\section{Preliminaries}
\label{sec:prelims}
If not specified, we use the $\ell_2$-norm, i.e., $\|\cdot\|:=\|\cdot\|_2$.

Let $v^*$ be the optimal value of \eqref{lp}, which we assume to be finite.
For $x\in P$, we denote its optimality gap as $g(x) := \pr{c}{x} - v^*$.
Given an elementary vector $h\in \mathcal{E}(A)\setminus \R^n_{\geq 0}$, we define $\mathrm{aug}_P(x,h) := x+\alpha h$ where $\alpha=\max\{\bar \alpha\in\R:x+\bar \alpha h\in P\}$.

For $g\geq 0$, we define $\tilde{x}^\mcp(g):=(1/n)\cdot\sum_{i\in[n]}\mathrm{argmax}\{x_i:x\in P_g\}$ as a feasible approximation of the max central path. 
Note that $\tilde{x}^\mcp(g)\in P$ by convexity.
Moreover, a simple averaging argument yields $\mx(g)/n \le \tilde{x}^\mcp(g)\le \mx(g) $.

\subsection{The \textsc{Ratio-Circuit} Oracle}\label{prelim:ratio-circuit}

We will use the following oracle to perform circuit augmentations.

\begin{subroutineBox}[label={subroutine:ratio-circuit}]{\textsc{Ratio-Circuit}$(A,u,v,w)$}
For a matrix $A\in\R^{m\times n}$, costs $u\in\R^n$ and weights $v,w\in(\R^n_{\geq 0}\cup\{\infty\})^n$, the output is a basic optimal solution to the system:
\begin{equation}
    \label{sys:min_ratio}
    \min\; \pr{u}{z} \quad \mathrm{s.t.} \quad Az =\zero,\, \quad 
    \pr{v}{z^+} + \pr{w}{z^-} \leq 1\, ,
\end{equation} 
and an optimal solution to the following dual program:
\begin{equation}
    \label{sys:min_ratio_dual}
    \max\; -\lambda \quad \mathrm{s.t.}\quad A^\top y+ s=u,\, \quad
    -\lambda v \le s\le \lambda w\, .
\end{equation}
\end{subroutineBox}

Recall that we use the convention $ab = 0$ if $a = \infty$ and $b = 0$.
Note that \eqref{sys:min_ratio} can be reformulated as an LP using auxiliary variables:
\begin{equation}
    \label{sys:min_ratio_lp}
    \begin{aligned}
        &\min\; \pr{u}{z^+} - \pr{u}{z^-} \\
        &\mathrm{s.t.} \quad Az^+ - Az^- =\zero\, , \\
        &\qquad\; \pr{v}{z^+} + \pr{w}{z^-} \leq 1\, , \\
        &\qquad\; z^+,z^-\geq \zero\, ,
    \end{aligned}
\end{equation}
and its dual LP can be equivalently written as \eqref{sys:min_ratio_dual}.
For a feasible solution $(z^+, z^-)$ to \eqref{sys:min_ratio_lp}, we define $z:= z^+ - z^-$ as the corresponding solution to \eqref{sys:min_ratio}.
By a basic optimal solution to \eqref{sys:min_ratio}, we mean the corresponding solution of a basic optimal solution to \eqref{sys:min_ratio_lp}.

\begin{restatable}{lemma}{ratiocircuitelementary}\label{ratio-circuit-elementary}
If \eqref{sys:min_ratio} is bounded, then a basic optimal solution is either $\zero$ or an elementary vector $z\in \elementary{A}$ that minimizes 
\[\frac{\pr{u}{z}}{\pr{v}{z^+}+\pr{w}{z^-}}.\]
\end{restatable}

The following standard lemma gives the guarantees of Wallacher's rule~\cite{W89} (see, for example, \cite[Lemma 2.5]{DKNV24} for a proof).

\begin{lemma}\label{ratio-geometric}
Given $x\in P$, let $h$ be the optimal solution returned by $\textsc{Ratio-Circuit}(A,c,\zero,1/x)$.
\begin{enumerate}
    \item If $\pr{c}{h} = 0$, then $x$ is optimal to \eqref{lp}.
    \item If $\pr{c}{h}<0$, then letting $x':= \aug_P(x,h)$, we have $1\leq \alpha \leq n$ for the augmentation step size and $g(x') \leq (1-1/n)g(x)$.
\end{enumerate}
\end{lemma}

\begin{corollary}\label{iterated-geometric-progress}
    Let $x\in P$ with optimality gap $g$, and let $0\leq g'\leq g$.
    Starting from $x$, we can reach some point $x'\in P$ with $g(x')\leq g'$ via a circuit walk of length $\lceil n\log(g/g') \rceil$ by repeatedly applying \Cref{ratio-geometric}.
\end{corollary}

\subsection{Properties of the Max Central Path}\label{prelim:mcp-properties}

We collect a few useful properties of the max central path.
Let $\ms$ denote the max central path of the dual LP \eqref{dlp}.
The theorem below shows that $\mx$ and $\ms$ are approximately central.

\begin{theorem}[{\cite[Theorem~4.2]{ADL25}}]\label{thm:centrality}
For any $g\geq 0$, we have
\[g\leq \mx_i(g)\ms_i(g) \leq 2g \qquad \forall i\in [n].\]
\end{theorem}

The upper bound above is immediate by the definition of $\mx$ and $\ms$.
For a fixed $i\in [n]$ and $g\geq 0$, letting $x^{(i)}$ and $s^{(i)}$ be feasible primal and dual solutions realizing $\mx_i(g)$ and $\ms_i(g)$ respectively, we have
\[\mx_i(g)\ms_i(g) = x^{(i)}_i s^{(i)}_i \leq \langle x^{(i)}, s^{(i)} \rangle = \langle x^{(i)}, s^* \rangle + \langle x^*, s^{(i)} \rangle \leq 2g,\]
where $x^*$ and $s^*$ are optimal primal and dual solutions.

The following lemma compares any two points of the max central path. 
The upper bound is due to the monotonicity of $\mx$, while the lower bound follows from the concavity and nonnegativity of $\mx$.

\begin{lemma}\label{lem:mcp_bounds}
For any $g'\geq g\geq 0$, we have
\[\frac{g}{g'}\mx(g') \leq \mx(g) \leq \mx(g').\]
\end{lemma}

Next, we define what it means for a segment of the max central path to be \emph{polarized}.

\begin{definition}
Given $\gamma\in (0,1]$ and $0\leq g_0< g_1$, we say that the max central path segment $\{\mx(g):g\in [g_0,g_1]\}$ is \emph{$\gamma$-polarized} if there exists a partition $(B,N)=[n]$ such that for every $g\in [g_0,g_1]$,
\begin{equation}\label{eq:polarized-gamma}
\begin{aligned}
    \gamma \mx_i(g_1) &\leq \mx_i(g) \leq \mx_i(g_1) \qquad \quad\;\;\forall i\in B \\
    \frac{g}{g_1} \mx_i(g_1) &\leq \mx_i(g) \leq \frac{g}{\gamma g_1} \mx_i(g_1) \qquad \forall i\in N.
\end{aligned}
\end{equation}
We say that the interval $[g_0, g_1]$ is \emph{$\gamma$-polarized}, and call $(B,N)$ its \emph{polarized partition}.
\end{definition}

Note that the upper bound for $B$ and the lower bound for $N$ always holds by \Cref{lem:mcp_bounds}. So, the interesting part of the definition is the lower bound for $B$ and the upper bound for $N$.

The following lemma decomposes the max central path into $O(\sum_{i\in [n]}\slc_{\eta}(\mx_i))$ many polarized segments.
For concreteness, we use $\eta = 1/2$ throughout the paper.

\begin{lemma}[{\cite[Lemma~4.5]{ADL25}}]\label{lem:polarization}
Let $x\in P$ with optimality gap $g$.
For any $\eta\in (0,1]$, there exist points $0 = g_0 < g_1 < \dots g_r = g$ with $r \leq 2\sum_{i\in [n]}\slc_{\eta}(\mx_i,g)$ such that each interval $[g_j, g_{j+1}]$ is $(\eta/2)$-polarized.
\end{lemma}

\subsection{The Lifting Operator and Its Singular Subspaces}

Recall the definition of the lifting operator in the introduction.

\begin{definition}[Lifting Map and Operator] \label{def:lift}
    Given a linear subspace $W\subseteq\R^n$ and a partition $(I,J)$ of $[n]$, we define the lifting map $L_I^W:\pi_I(W)\to W$ as 
    \begin{equation*}
        L_I^W(x):=\arg\min_{w\in W}\{\|w\|: w_I=x\}.
    \end{equation*}
    Further, we define the lifting operator $\ell_I^W:\pi_I(W)\to \pi_J(W)$ as $\ell_I^W(x):=\pi_J(L_I^W(x))$.
\end{definition}

Since $\ell^W_I$ is a linear operator, it admits a singular value decomposition.

\begin{definition}[Singular value decomposition]
    Let $U\subseteq\R^n$ and $V\subseteq\R^m$ be linear subspaces.
    Let $T:U\to V$ be a linear operator and $\mathcal{M}(T)$ be its matrix representation.
    Then, $T$ admits a \emph{singular value decomposition} (SVD)
    \begin{equation*}
        \mathcal{M}(T)=\sum_{i=1}^{\rk(T)}\sigma_i(T)v_iu_i^\intercal,
    \end{equation*}
    where $v_1,\dots,v_{\rk(T)}\in V$ and $u_1,\dots,u_{\rk(T)}\in U$ are orthonormal vectors in their respective subspaces, and $\sigma_1(T)\geq \sigma_2(T) \geq \dots \geq \sigma_{\rk(T)}(T)>0$.
    We also set $\sigma_i(T):=0$ for all $\rk(T)<i\leq \dim(U)$, and $\sigma_0(T):=\infty$ (with the convention that $1/\infty=0$).
    We denote by $\sigma_{\min}(T):=\sigma_{\dim(U)}(T)$ the smallest singular value of $T$.

    Given an interval $I\subseteq\R_{\geq 0}$, we define $C_\sigma(T,I):=|\{i\in[\dim(U)]:\sigma_i(T)\in I\}|$ to count the number of singular values of $T$ in $I$.
    Note that $\sigma_i(T) \geq \alpha$ if and only if $C_{\sigma}(T,[\alpha,\infty)) \geq i$.
\end{definition}

For the algorithmic implementation in Section~\ref{sec:implementation}, we assume that we can compute the SVD exactly for the sake of simplicity.
Our framework can also handle approximate SVD computations using the techniques developed in ~\cite{ADL25}.

The concept of singular subspaces will be crucial to our analysis.

\begin{definition}[Singular subspaces]
    Let $T:U\to V$ be a linear operator.
    A linear subspace $S\subseteq U$ is called a \emph{singular subspace for $T$} if 
    \begin{equation*}
        \sigma_1(T\vert_S)\le\sigma_{\mathrm{dim}(U)-\mathrm{dim}(S)+1}(T).
    \end{equation*}
    By the min-max principle for singular values, the singular subspaces for $T$ are exactly the subspaces that attain the minimum in
    \begin{align}
        \sigma_i(T)&=\min_{\substack{\mathrm{dim}(S)\ge\mathrm{dim}(U)-i+1\\ S\subseteq U}}\ \max_{x\in S\setminus\{\zero\}}\frac{\|T(x)\|}{\|x\|}\label{sv_minmax}
    \end{align}
    for each $i\in [\dim(U)]$.
\end{definition}

The following lemma shows that projecting a vector onto a singular subspace for $\ell^W_N$ can only change the entries in $N$ in a bounded way. 
A similar statement has previously been used by Allamigeon et al. in the context of interior point methods, see~\cite[Lemma~7.8]{ADL25}.

\begin{restatable}{lemma}{projections}\label{projections}
    Given a linear subspace $W\subseteq \R^n$, let $1\le k\le \dim(\pi_N(W))$ and $d:=\dim(\pi_N(W))-k$.
    Let $V_d\subseteq\pi_N(W)$ be a singular subspace for $\ell_N^W$ of dimension $d$ and $W_d:=L_N^W(V_d)$.
    For any $z\in W$, if $z^p:=\arg\min_{w\in W_d}\|z_N-w_N\|$, then
    \begin{equation*}
        \|z_N-z_N^p\|\le \frac{\|z_B\|}{\sigma_k}.
    \end{equation*}
\end{restatable}

The following lemma describes the behavior of the singular values of the lifting operator under subspace rescaling.
\begin{lemma}[{\cite[Lemma~7.12]{ADL25}}]\label{lemma:singular-value-scaling}
    Let $W\subseteq \R^n$ be a subspace, and let $(B,N)$ be a non-trivial partition of $[n]$.
    Let $y\in\R^n_{>0}$, and let $W':=\mathrm{diag}(y)^{-1}W$.
    We define $\ell:=\ell_N^{W}$ and $\ell':=\ell_N^{W'}$ with respective singular value decomposition $\sigma:=\sigma(\ell)$ and $\sigma':=\sigma(\ell')$.
    Then, we have that 
    \begin{equation*}
        \frac{1}{\|y_B^{-1}\|_\infty\|y_N\|_\infty}\sigma'\le\sigma\le\|y_B\|_\infty\|y_N^{-1}\|_\infty\sigma'.
    \end{equation*}
\end{lemma}

\section{The Circuit Diameter Bound}\label{sec:bound}
In this section, we present an `existential' circuit augmentation algorithm (Algorithm~\ref{alg:non-det}).
Given an initial feasible solution $x^{(0)}$, it solves \eqref{lp} in $O(n^2\log n\cdot\sum_{i\in[n]}\slc_{1/2}(x_i^\mcp,g(x^{(0)})))$ circuit steps, assuming that it has access to the max central path and its polarized decomposition.
This latter assumption is what makes the algorithm `existential', which we will later remove in \Cref{sec:implementation}.
The goal of this section is to prove a weaker bound of $O(n^3\log n\cdot\sum_{i\in[n]}\slc_{1/2}(x_i^\mcp,g(x^{(0)})))$.
We defer the stronger amortized bound to \Cref{sec:amortization}.

Throughout this section, we will use the following six functions:
\begin{align*}
    &f_1(n):=32n^2\cdot (3n)^{5n}\cdot f_6(n)&&f_2(n):=64n^{2.5}\\ 
    &f_3(n,p):=(3n)^{5p}\cdot f_6(n)&&f_4(n):=2n \\ 
    &f_5(n):= 2n^2  &&f_6(n):= 3785n^6. 
\end{align*}

\begin{algorithm}[!h]\label{alg:non-det}
\caption{Existential circuit augmentation algorithm}
\SetKwInOut{Input}{Input}
\SetKwInOut{Output}{Output}
\Input{Bounded instance of \eqref{lp} with constraint matrix $A\in\R^{m\times n}$, feasible polyhedron $P\subseteq\R^n$, feasible solution $x^{(0)}\in P$ with gap $g_r$, and polarized decomposition $[g_0,g_1]\cup[g_1,g_2]\cup\dots\cup [g_{r-1},g_r]$ of $x^\mcp$}
\Output{Optimal solution $x^*$}
$x\gets x^{(0)}$\;
\While{$x$ is not optimal}{\label{pc:main-while}
    Find $j\ge0$ s.t. $g(x)\in (g_j,g_{j+1}]$ with associated polarized partition $(B,N)$\;
    Compute the lifting operator $\ell_N^{W_j}$ where $W_j := {\rm diag}(\mx(g_{j+1}))^{-1}\mathrm{ker(A)}$\; 
    Compute the singular values $\sigma_1\geq \sigma_2 \geq \dots\geq \sigma_{\dim(\pi_N(W_j))}$ of $\ell_N^{W_j}$\;
    $k\gets C_\sigma(\ell_N^{W_j},[g_{j+1}/g(x),\infty))$\;
    \lIf{$k=0$}{$\hat g\gets g_j$}\lElse{$\hat g\gets\max\{g_j,g_{j+1}/\sigma_k\}$}
    \If{$k<\dim(\pi_N(W_j))$ and $f_1(n)\cdot \hat g<g(x)\leq g_{j+1}/(f_2(n)\cdot \sigma_{k+1})$}{
        $\bar x\gets \tilde x^{\mcp}(g(x))$\;
        $y^{(0)} \gets \tilde x^{\mcp}(n\hat g)$\;
        $x\gets \textsc{Long-Steps}(x, \bar x, y^{(0)},N, \hat g)$\;
    }
    
    Compute elementary vector $z\in\elementary{A}$ as solution to \textsc{Ratio-Circuit}$(A,c,\zero,1/x)$\;
    $x\gets\aug_P(x,z)$\;
}
\Return $x$\;
\end{algorithm}

The high-level description of Algorithm~\ref{alg:non-det} is quite simple.
In every iteration, we first identify the polarized interval $[g_j,g_{j+1}]$ containing the current optimality gap $g(x)$, along with its polarized partition $(B,N)$.
Then, we compute the lifting operator $\ell^{W_j}_N$ and its singular values $\sigma_1 \geq \sigma_2 \geq \dots \geq \sigma_{\dim(\pi_N(W_j))}$, where $W_j = {\rm diag}(\mx(g_{j+1}))^{-1}\ker(A)$.

Let $a:=C_\sigma(\ell_N^{W_j},[g_{j+1}/ g_j,\infty))+1$ and $b:=C_\sigma(\ell_N^{W_j},[1,\infty))$. 
Note that if $a \leq \dim(\pi_N(W))$, then $g_{j+1}/\sigma_a$ is the first singular value breakpoint above $g_j$.
Similarly, if $b \geq 1$, then $g_{j+1}/\sigma_b$ is the last singular value breakpoint below $g_{j+1}$. 
If $a\leq b$, we further subdivide the interval $[g_j,g_{j+1}]$ into subintervals based on the singular values, i.e., 
$$\left[g_j, \frac{g_{j+1}}{\sigma_a}\right], \left[\frac{g_{j+1}}{\sigma_a}, \frac{g_{j+1}}{\sigma_{a+1}}\right], \dots, \left[\frac{g_{j+1}}{\sigma_b}, g_{j+1}\right].$$
Let $k:=C_\sigma(\ell_N^{W_j},[g_{j+1}/g(x),\infty))$.
If $k = 0$, then $g(x)$ lies in the subinterval with left endpoint $g_j$.
Otherwise, $g(x)$ lies in the subinterval with left endpoint $g_{j+1}/\sigma_k$.
So, we set $\hat g$ as the corresponding left endpoint.

Our goal is to cross $\hat g$ quickly via circuit steps.
If the current gap $g(x)$ is multiplicatively close to $\hat g$, i.e., $g(x)\leq f_1(n)\cdot \hat g$, then this can be achieved by repeatedly applying Wallacher's rule according to \Cref{iterated-geometric-progress}.
It turns out that this is always the case when $k=\dim(\pi_N(W_j))$.
On the other hand, if $g(x)$ is multiplicatively far away from $\hat g$, then we need a different circuit step.
In this scenario, the algorithm first runs Wallacher's rule to move away from the right endpoint $g_{j+1}/\sigma_{k+1}$.
Once $g(x)\leq g_{j+1}/(f_2(n)\cdot \sigma_{k+1})$, then it uses the \textsc{Long-Steps} subroutine to approach $\hat g$.

\begin{algorithm}[!h]\label{alg:long-steps}
\caption{\textsc{Long-Steps}}
\SetKwInOut{Input}{Input}
\SetKwInOut{Output}{Output}
\Input{Feasible solution $x\in P$, points $\bar x, y^{(0)}\in P$ satisfying Properties \eqref{cond:x_gap}, \eqref{cond:x_lbound}, \eqref{cond:y_gap}, \eqref{cond:y_lbound}, subset of coordinates $N\subseteq [n]$, target gap $\hat g$}
\Output{Feasible solution $x\in P$}
Let $L$ be the ray containing $[y^{(0)}, \bar x]$, where $L(g)$ is the point on $L$ with optimality gap $g\geq 0$\;
$p\gets0$\;
\While{$f_1(n)\cdot \hat g<g(x)$}{\label{pc:inner-while}
    $S_p \gets \{i\in N: x_i \leq 3ny^{(p)}_i\}$\;
    $y^{(p+1)} \gets L(g)$ for $g\geq 0$ minimal such that $L_N(g)\geq 4ny^{(p)}_N$\;
    Compute elementary vector $z\in\elementary{A}$ as a solution to~\eqref{sys:long_steps_lp}\;
    $x\gets\aug_P(x,z)$\;
    $p\gets p+1$\;
}

\Return $x$\;
\end{algorithm}

The \textsc{Long-Steps} subroutine relies on two key points $\bar x,y^{(0)}\in P$, which satisfy the following properties:

\medskip
\begin{minipage}{0.49\textwidth}
\begin{enumerate}
    \item\label{cond:x_gap} $g(\bar x)\leq f_5(n)\cdot g(x^{(t)})$; 
    \item\label{cond:x_lbound} $\bar x_i \geq \frac{\mx_i(g(x^{(t)}))}{f_5(n)}$ for all $i\in N$.
\end{enumerate}
\end{minipage}
\begin{minipage}{0.49\textwidth}
\begin{enumerate}[label=(\alph*)]
    \item\label{cond:y_gap} $g(y^{(0)})\leq f_6(n)\cdot \hat g$;
    \item\label{cond:y_lbound} $y^{(0)}_i\geq \frac{1}{\sigma_k}$ for all $i\in N$.
\end{enumerate}
\end{minipage}
\medskip

Observe that the choice of $\bar x := \tilde x^\mcp(g(x^{(t)}))$ and $y^{(0)} := \tilde x^\mcp(n\hat g)$ in Algorithm~\ref{alg:non-det} actually satisfies these properties for smaller values of $f_5(n)$ and $f_6(n)$, i.e., $f_5(n) = f_6(n) = n$.
The reason we set them larger now is to allow for approximation errors when we implement the algorithm in \Cref{sec:implementation}.
As we will not have access to the max central path $\mx$ or its polarized decomposition, we could only approximate these quantities. 

Let $L:\R_{\geq 0}\to \R^n$ be the ray containing the line segment $[y^{(0)}, \bar x]$, where $L(g)$ is the point on $L$ whose optimality gap is equal to $g$.
As long as $g(x)$ is multiplicatively far from $\hat g$, \textsc{Long-Steps} repeats the following procedure for $p\in \Z_{\geq 0}$, starting with $p=0$.
First, we identify the set of \emph{small} variables in $N$ as
\begin{equation}\label{eq:small_coords}
    S_p:=\left\{i\in N:x_i\le 3ny^{(p)}_i\right\}. 
\end{equation}
We also set $y^{(p+1)} \in \R^n$ as 
\begin{equation}\label{eq:moving_target}
    y^{(p+1)} := L(g)
\end{equation}
where $g\geq 0$ is the smallest value such that $L_N(g) \geq 4ny^{(p)}_N$.
Then, we define the following linear program, which we call the \emph{Long Step Circuit LP~\eqref{sys:long_steps_lp}}. 
\begin{equation}
    \label{sys:long_steps_lp}
    \tag{LSC}
    \begin{aligned}
        &\min\; \pr{c}{z^+}-\pr{c}{z^-}\\
        &\mathrm{s.t.} \quad Az^+ - Az^- =\zero\, , \\
        &\qquad\; \sum_{i\in S_p}\frac{z_i^+}{2y^{(p+1)}_i} + f_4(n)\sum_{i\in B}\frac{z_i^-}{x_i}+\sum_{i\in N\setminus S_p}\frac{z_i^-}{x_i}\leq 1\, , \\  
        &\qquad\; z^+,z^-\geq \zero\, ,\ z_i^-=0\quad\forall i\in S_p\, .
    \end{aligned}
\end{equation}
Note that it can be solved using the \textsc{Ratio-Circuit} oracle.
After obtaining an elementary vector $z\in \mathcal{E}(A)$ as a solution to \eqref{sys:long_steps_lp}, we augment our current point $x$ with it.
This finishes the description of an iteration of \textsc{Long-Steps}.
As we will see in the next section, there are at most $n$ such iterations.

\subsection{Traversing a Long Singular Value Subinterval}\label{sec:long}

In this subsection, we show that Algorithm~\ref{alg:non-det} traverses each singular value subinterval in $O(n^2\log n)$ calls to \textsc{Ratio-Circuit}.
We first set up some notation for the analysis.
For $t\geq 0$, let $x^{(t)}$ be the iterate of Algorithm~\ref{alg:non-det} after $t$ calls to \textsc{Ratio-Circuit}, including the ones made in \textsc{Long-Steps}.
Fix $t\geq 0$, and let $[g_j,g_{j+1}]$ be the polarized interval containing $g(x^{(t)})$, with its polarized partition $(B,N) = [n]$.
For convenience, we may assume that $x^\mcp_i(g_{j+1}) = 1$ for all $i\in [n]$. 
This can be achieved by multiplying $c_i$ and the $i$th column of $A$ by $1/x^\mcp_i(g_{j+1})$.
As $x^\mcp_i(g)$ scales by $1/x^\mcp_i(g_{j+1})$ for all $g\geq 0$, its straight line complexity remains the same.
Note that if $x^\mcp_i(g_{j+1})=0$, then $\mx_i(g) = 0$ for all $g\geq 0$.
In this case, $P = P\cap \{x\in \R^n:x_i = 0\}$ so we can eliminate the variable $x_i$. 
We may also assume that $g_{j+1} = 1$ by multiplying $c$ by $1/g_{j+1}$.
Denoting $\tilde g := g_j$, we have $g(x^{(t)})\in [\tilde g, 1]$.

Consider the lifting operator $\ell_N^W$ where $W:=\ker(A)$, along with its singular values $\sigma_1\geq \sigma_2 \geq \dots \geq \sigma_{\mathrm{dim}(\pi_N(W))}$.
Let $k:= C_\sigma(\ell^W_N, [1/g(x^{(t)}),\infty))$ and assume that $k<\dim(\pi_N(W))$; we will deal with the case $k=\dim(\pi_N(W))$ in \Cref{sec:iterations}.
Then, $g(x^{(t)})$ lies in the singular value subinterval $[\hat g, 1/\sigma_{k+1}]$, where $\hat g = \max\{\tilde g, 1/\sigma_k\}$.
We call the subinterval \emph{long} if $f_1(n)f_2(n)\cdot \hat g< 1/\sigma_{k+1}$, and \emph{short} otherwise.
If the subinterval is short, then according to \Cref{iterated-geometric-progress}, we can traverse it in $O(n^2\log n)$ calls to \textsc{Ratio-Circuit} with Wallacher's rule.
So, we may assume that the subinterval is long.
We may also assume that $g(x^{(t)})$ lies deep in the subinterval, i.e.,
\begin{equation}\label{eq:deep}
    f_1(n)\cdot \hat g<g(x^{(t)})\leq \frac{1}{f_2(n)\cdot \sigma_{k+1}}.
\end{equation}
Noting that $[\hat g, f_1(n) \cdot \hat g]$ and $[1/(f_2(n) \sigma_{k+1}),1/\sigma_{k+1}]$ are \emph{short}, it suffices to show that this deep part of the subinterval $[\hat g, 1/\sigma_{k+1}]$ can be traversed quickly.

By \Cref{lem:polarization} with $\eta = 1/2$, we have
\begin{align}
    \frac{1}{4n}&\le\tilde x_i^\mcp(g)\le 1 &\forall i\in B,\ g\in [\tilde g,1]&\label{approx-B}\\
    \frac{g}{n}&\le\tilde x_i^\mcp(g)\le 4g &\forall i\in N,\ g\in [\tilde g,1]&.\label{approx-N}
\end{align}
The following lemma upper bounds the optimality gap of a feasible point in terms of its 1-norm. 

\begin{restatable}{lemma}{gaponN}\label{gap-on-N}
    For any feasible solution $x$ to \eqref{lp}, its optimality gap is
    $g(x)\le 8\tilde g\|x_B\|_1 + 2\|x_N\|_1$.
\end{restatable}
\begin{proof}
    Let $s^*$ be an optimal solution to the dual LP~\eqref{dlp}.
    The optimality gap of $x$ is given by
    \begin{equation*}
        g(x)=\pr{x}{s^*}\le\sum_{i\in [n]}x_is_i^\mcp(\tilde g)\le2\sum_{i\in[n]}\frac{x_i\tilde g}{x_i^\mcp(\tilde g)}\le 8\sum_{i\in B}x_i\tilde g+2\sum_{i\in N}x_i= 8\tilde g \|x_B\|_1 + 2\|x_N\|_1.
    \end{equation*}
    The second inequality follows from the upper bound in \Cref{thm:centrality}, while the third inequality is due to our assumption that $[\tilde g, 1]$ is $\frac14$-polarized and $\mx(1) = \one$.
\end{proof}

Since $g(x^{(t)})\gg \tilde g$ and $x^{(t)}\leq \mx(1) = \one$, \Cref{gap-on-N} shows that the optimality gap of $x^{(t)}$ is essentially determined by $\|x^{(t)}_N\|_1$.
This is suggestive of decreasing the coordinates in $N$ in order to approach $\tilde g$.
Recall the definition of small coordinates in \eqref{eq:small_coords}.
Once a coordinate becomes small, we want it to remain small.
More formally, for $r\geq t$ and $p\in \Z_{\geq 0}$, let us denote
\[S^{(r)}_p := \left\{i\in N:x_i^{(r)}\le 3ny^{(p)}_i\right\}.\]
Then, we want the set $S^{(r)}_p$ to be monotone nondecreasing with $r$ and $p$.
This sentiment is reflected in the constraint of \eqref{sys:long_steps_lp}, where any increase in the small coordinates is penalized.

The next proposition shows that by taking circuit steps given by basic optimal solutions to \eqref{sys:long_steps_lp}, the set of small coordinates grows or we exit the deep part of the subinterval. 
Thus, \textsc{Long-Steps} terminates in at most $|N|\leq n$ iterations.

\begin{proposition}\label{scaling-up}
    Let $x^{(t)}$ be an iterate of the circuit walk satisfying Condition~\eqref{eq:deep}.
    Let $\bar x$ and $y^{(0)}$ be feasible solutions to \eqref{lp} with Properties~\ref{cond:x_gap}, \ref{cond:x_lbound} and \ref{cond:y_gap}, \ref{cond:y_lbound} respectively.
    For $r\geq t$, if $x^{(r)}$ satisfies Condition~\eqref{eq:deep} and $S^{(r)}_p\neq N$ for some $p\in\Z_{\geq 0}$, then any basic optimal solution $z^*\in\mathcal{E}(A)$ to~\eqref{sys:long_steps_lp} with corresponding next iterate $x^{(r+1)}:=\mathrm{aug}_P(x^{(r)},z^*)$ satisfies 
    \begin{itemize}
        \item $g(x^{(r+1)})<g(x^{(r)})$ and $S^{(r)}_p\subsetneq S^{(r+1)}_{p+1}$ \emph{or}
        \item $g(x^{(r+1)})\le f_1(n)\cdot\hat g$.
    \end{itemize}
\end{proposition}

Fix $r\geq t$ and $p\in \Z_{\geq 0}$.
In order to prove \Cref{scaling-up}, we carefully construct a target point $y^\circ$ such that $z^\circ:= y^\circ - x^{(r)}$ is a good solution to \eqref{sys:long_steps_lp}.
First, we define a shifted iterate 
\begin{equation}\label{eq:conv_comb}
    \hat x:=\left(1-\frac{1}{f_4(n)}\right)x^{(r)}+\frac{1}{f_4(n)}\cdot \tilde x^\mcp(g(x^{(r)})),
\end{equation}
i.e., we slightly push $x^{(r)}$ in direction of the max central path. 
Then, we construct the target as 
\begin{equation}\label{eq:target}
    y^\circ:=\hat x+(\ell_N^W(\Pi_{V_d}(y^{(p+1)}_N-\hat x_N)),\Pi_{V_d}(y^{(p+1)}_N-\hat x_N)),
\end{equation}
where $V_d$ is a singular subspace for $\ell_N^W$ of dimension $d:=\dim(\pi_N(W))-k$, and $\Pi_{V_d}$ is the orthogonal projection from $\pi_N(W)$ onto $V_d$.
We now explain the intuition behind this construction. 
As indicated in \Cref{scaling-up}, the goal is to zero out a coordinate in $N\setminus S_p^{(r)}$, while keeping the coordinates in $B$ and $S_p^{(r)}$ positive, and the coordinates in $S_p^{(r)}$ small.
If we simply used $y^{(p+1)}$ as the target, the entries of $x^{(r)}_{B\cup S_p^{(r)}}$ that are larger than $y^{(p+1)}_{B\cup S_p^{(r)}}$ may be set to zero first. 
To prevent the entries of $x^{(r)}_B$ from hitting zero, one could pick the target $y^\circ$ such that $y^\circ_B\approx x^{(r)}_B$.
This can be achieved by first projecting $y^{(p+1)}_N - x^{(r)}_N$ onto the singular subspace $V_d$, and then lifting it to the coordinates in $B$, i.e., setting $y^\circ$ to 
\[x^{(r)} + (\ell^W_N(\Pi_{V_d}(y^{(p+1)}_N - x^{(r)}_N)), \Pi_{V_d}(y^{(p+1)}_N - x^{(r)}_N)).\]
However, this only yields additive closeness, which does not work for the entries of $x^{(r)}_B$ that are tiny.
To fix this, we first push these coordinates up a little bit using the convex combination \eqref{eq:conv_comb}.

The next consideration is to prevent the entries of $x^{(r)}_{S_p^{(r)}}$ from hitting zero.
We achieve this by picking the target such that $y^\circ_{S_p^{(r)}}\geq x^{(r)}_{S_p^{(r)}}$.
As the error on $N$ incurred by the projection $\Pi_{V_d}$ is proportional to $1/\sigma_{k+1}$ (\Cref{proximity-on-N}), this is guaranteed by the definition of $y^{(p+1)}$ \eqref{eq:moving_target} and Property \ref{cond:y_lbound} of $y^{(0)}$.
For the same reason, $y^\circ_{S_p^{(r)}}$ is not  much larger than $y^{(p+1)}_{S_p^{(r)}}$.
Thus, the entries of $x^{(r)}_{S_p^{(r)}}$ do not increase too quickly, which ensures that the set of small coordinates $S_p^{(r)}$ is monotone with respect to inclusion.

Recall that we defined $L:\R_{\ge0}\to\R^n$ as the ray containing the line segment $[y^{(0)},\bar x]$ 
We first lower bound the gradient of $L_N$ with respect to the optimality gap.
\begin{lemma}\label{lem:gradient}
For every $i\in N$, we have 
\[\frac{dL_i}{dg} := \frac{\bar{x}_i-y^{(0)}_i}{g(\bar{x})-g(y^{(0)})}\geq \frac{1}{2f_5(n)^2}.\]
\end{lemma}

\begin{proof}
Fix $i\in N$.
By Properties~\ref{cond:x_gap}, \ref{cond:x_lbound} and \ref{cond:y_gap},
\[\frac{dL_i}{dg} = \frac{\bar{x}_i-y^{(0)}_i}{g(\bar{x})-g(y^{(0)})} \geq \frac{\mx_i(g(x^{(t)}))/f_5(n) - \mx_i(f_6(n)\cdot \hat g)}{f_5(n)\cdot g(x^{(t)})}.\]
Recall that $[\tilde g, 1]$ is $\frac14$-polarized and $\mx(1) = \one$.  
Since $f_6(n)\cdot \hat g\in [\tilde g, 1]$, the above is lower bounded by
\[ \frac{g(x^{(t)})/f_5(n) - 4 f_6(n)\cdot\hat g}{f_5(n)\cdot g(x^{(t)})} \geq \frac{1}{2f_5(n)^2}\]
as desired.
Indeed, the last inequality holds by  \eqref{eq:deep} because
\[4 f_6(n)\cdot \hat g \leq \frac{f_1(n)}{2f_5(n)} \cdot \hat g < \frac{g(x^{(t)})}{2f_5(n)}. \qedhere\]
\end{proof}

The following lemma shows the feasibility of $y^{(p)}$ to \eqref{lp} and bounds its optimality gap.

\begin{lemma}\label{lem:gap}
For every $p\in \{0,1,\dots,n\}$, we have
\[g(y^{(0)})\leq g(y^{(p)}) \leq f_3(n,p) \cdot \hat g.\]
Consequently, $y^{(p)}$ is feasible to \eqref{lp}.
\end{lemma}

\begin{proof}
We prove the first statement by induction on $p$.
The base case $p=0$ is clear by construction.
Let us assume that the lemma holds for some $p\geq 0$, and consider the case $p+1$.
By \Cref{lem:gradient}, $y^{(p+1)}$ exists and $g(y^{(p+1)}) \geq g(y^{(p)}) \geq g(y^{(0)})$.
Moreover, there exists an $i\in N$ such that $y^{(p+1)}_i = 4ny^{(p)}_i$.
Applying \Cref{lem:gradient} yields
\begin{align*}
g(y^{(p+1)}) &\leq 8n^4\left( y^{(p+1)}_i - y^{(p)}_i \right) + g(y^{(p)}) \\
&\leq 32n^5 y^{(p)}_i + g(y^{(p)}) \\
&\leq 32n^5\mx_i(f_3(n,p)\cdot\hat g) + f_3(n,p)\cdot\hat g \tag{by inductive hypothesis} \\
&\leq 4\cdot 32n^5 \cdot f_3(n,p)\cdot\hat g + f_3(n,p)\cdot\hat g \tag{$[\tilde g, 1]$ is $\frac14$-polarized and $\mx(1) = \one$} \\
&\leq f_3(n,p+1)\cdot\hat g.
\end{align*}

To prove the second statement, it suffices to show that $f_3(n,n)\cdot \hat g\leq g(\bar x)$.
This is because $L(g)$ is feasible to \eqref{lp} for all $g\in [g(y^{(0)}), g(\bar x)]$.
First, we claim that $g(\bar x)/g(x^{(t)})\geq 1/(4f_5(n))$.
Suppose otherwise for a contradiction.
Then, for every $i\in N$,
\[\bar x_i \leq \mx_i(g(\bar x)) < 4 \cdot \frac{1}{4f_5(n)} \cdot \mx_i(g(x^{(t)})) = \frac{\mx_i(g(x^{(t)}))}{f_5(n)},\]
where the strict inequality is due to $[g(x^{(t)})/(4f_5(n)),g(x^{(t)})]\subseteq [\tilde g, 1]$ being $\frac14$-polarized.
However, this contradicts Property~\ref{cond:x_lbound}.
Thus, we obtain
\[f_3(n,n)\cdot \hat g \leq \frac{f_1(n)\cdot \hat g}{4f_5(n)} < \frac{g(x^{(t)})}{4f_5(n)} \leq g(\bar x)\]
as desired, where the strict inequality is by \eqref{eq:deep}.
\end{proof}

The following lemmas state the crucial properties of the target point $y^\circ$.

\begin{lemma}\label{proximity-on-N}
    We have that $\|y^{(p+1)}_N-y_N^\circ\|\le\sqrt n/\sigma_k$.  
\end{lemma}
\begin{proof}
    If $k=0$, then $V_d=\pi_N(W)$, and $\Pi_{V_d}$ is the identity map.
    Therefore, $y_i^\circ=y^{(p+1)}_i$ for all $i\in N$.
    Otherwise, denoting $z:=y^{(p+1)}-\hat x$, we apply \Cref{projections} to obtain 
    \begin{equation*}
        \|y^{(p+1)}_N-y_N^\circ\|=\left\|z_N-\Pi_{V_d}(z_N)\right\|\le \frac{\|z_B\|}{\sigma_k}.
    \end{equation*}
    As long as $x^{(r)}$ satisfies Condition~\eqref{eq:deep}, we have $g(y^{(p+1)}) \leq g(x^{(r)}) = g(\hat x)$ by \Cref{lem:gap}.
    Since $[\tilde g, 1]$ is $\frac14$-polarized and $\mx(1) = \one$, both  $y^{(p+1)}_i$ and $\hat x_i$ are bounded from above by $1$ if $i\in B$. 
    Hence, $|z_i|\le1$ for all $i\in B$ and so $\|z_B\|\leq \sqrt{n}$.
\end{proof}

\begin{restatable}{lemma}{increaseonsmallN}\label{increase-on-small-N}
    Let $i\in N$. Then, $y^\circ_i\leq 2 y^{(p+1)}_i$.
    If $i\in S^{(r)}_p$, then $y^\circ_i\geq x^{(r)}_i$.
\end{restatable}

\begin{proof}
    Fix an $i\in N$.
    By \Cref{proximity-on-N}, we have that $|y^{(p+1)}_i-y_i^\circ|\le \|y^{(p+1)}_N-y_N^\circ\|\le \sqrt{n}/\sigma_k$. So,
    \[y^\circ_i \leq y^{(p+1)}_i + \left|y^{(p+1)}_i - y^\circ_i\right| \leq y^{(p+1)}_i + \frac{\sqrt{n}}{\sigma_k} \leq y^{(p+1)}_i + \frac{\sqrt{n}y^{(p+1)}_i}{4n} \leq 2y^{(p+1)}_i\]
    where the third inequality follows from \eqref{eq:moving_target} and Property~\ref{cond:y_lbound}.

    Next, suppose that $i\in S^{(r)}_p$. Then,
    \[y^\circ_i \geq y^{(p+1)}_i - \left|y^{(p+1)}_i - y^\circ_i\right| \geq y_i^{(p+1)} - \frac{\sqrt{n}}{\sigma_k} \geq 4n y_i^{(p)} - \sqrt{n}y^{(p)}_i \geq 3ny^{(p)}_i \geq x^{(r)}_i\]
    where the third inequality again follows from \eqref{eq:moving_target} and Property~\ref{cond:y_lbound}.
\end{proof}

\begin{restatable}{lemma}{consistencyinB}\label{consistency-in-B}
    Let $i\in B$. 
    Then, $(1-1/f_4(n))x_i^{(r)}\leq y^\circ_i\leq 2$.
\end{restatable}

\begin{proof}
    We begin by estimating the size of $\|\ell_N^W(\Pi_{V_d}(y^{(p+1)}_N-\hat x_N))\|$.
    As in the previous lemma, we denote $z:=y^{(p+1)}-\hat x$.
    Since $V_d$ is a singular subspace, it holds by \eqref{sv_minmax} that 
    \begin{align*}
        \|\ell_N^W(\Pi_{V_d}(z_N))\|&\le\sigma_{k+1}\cdot\|\Pi_{V_d}(z_N)\| \\
        &\le\sigma_{k+1}\cdot\|z_N\| \\
        &\le\sigma_{k+1}\cdot\left(\|y^{(p+1)}_N\|+\|\hat x_N\|\right) \\
        &\le\sigma_{k+1}\cdot\left(\|\mx_N(f_3(n,p+1)\cdot \hat g)\| + \|\mx_N(g(x^{(r)}))\|\right) \tag{by \Cref{lem:gap}}\\
        &\le\sigma_{k+1}\cdot2\|\mx_N(g(x^{(r)}))\| \tag{$f_3(n,p+1)\leq f_1(n)$ for all $p<n$} \\
        &\le\sigma_{k+1}\cdot8\sqrt{n}g(x^{(r)}) \tag{$[\tilde g, 1]$ is $\frac14$-polarized and $\mx(1) = \one$}\\
        &\le\frac{8\sqrt n}{f_2(n)}=\frac{1}{8n^2} \tag{by \eqref{eq:deep}}.
    \end{align*}

    Now, fix an $i\in B$.
    For the upper bound, we get
    \[y^\circ_i = \hat x_i + \ell_N^W(\Pi_{V_d}(z_N))_i \leq 1 + \frac{1}{8n^2} \leq 2.\]
    For the lower bound, we obtain
    \begin{align*}
        \frac{y^\circ_i}{x^{(r)}_i}&=\frac{\frac{1}{f_4(n)}\tilde x^\mcp_i(g(x^{(r)}))+(1-\frac{1}{f_4(n)})x^{(r)}_i+\ell_N^W(\Pi_{V_d}(z_N))_i}{x^{(r)}_i}\\
        &\ge1-\frac{1}{f_4(n)}+\frac{\frac{1}{4nf_4(n)}-\frac{1}{8n^2}}{x^{(r)}_i}\tag{by~\eqref{approx-B}}\\
        &=1-\frac{1}{f_4(n)}.
    \end{align*}
\end{proof}

\begin{proof}[Proof of \Cref{scaling-up}]
    Fix $r\geq t$ and $p\in \{0,1,\dots,n-1\}$ such that $S^{(r)}_p\neq N$.
    We first show that the LP~\eqref{sys:long_steps_lp} has a good solution if $x^{(r)}$ satisfies Condition~\eqref{eq:deep}.
    Let $z^\circ:=y^\circ-x^{(r)}$ be the direction to the target point.
    We claim that $z^\circ/n$ is a feasible solution to~\eqref{sys:long_steps_lp}.
    Clearly, $Az^\circ=0$. 
    By \Cref{increase-on-small-N}, we have 
    \[\frac{z_i^\circ/n}{2y^{(p+1)}_i}\le \frac{y^\circ_i/n}{2y^{(p+1)}_i} \leq \frac1n\] 
    for all $i\in N$, and $(z_i^{\circ})^-=0$ for all $i\in S_p^{(r)}$.
    By \Cref{consistency-in-B}, we have 
    \[\frac{(z_i^{\circ})^-/n}{x^{(r)}_i}\le \frac{(x^{(r)}_i-y^\circ_i)/n}{x^{(r)}_i} \leq \frac{1}{nf_4(n)}\]
    for all $i\in B$.
    Finally, we have 
    \[\frac{(z_i^{\circ})^-/n}{x_i^{(r)}}\le \frac{x^{(r)}_i/n}{x_i^{(r)}}= \frac1n\] 
    for all $i\in N\setminus S^{(r)}_p$, so $z^\circ/n$ is indeed feasible.

    Let $z^*$ be a basic optimal solution to~\eqref{sys:long_steps_lp}, which is an elementary vector. 
    Let $x^{(r+1)}:=\aug_P(x^{(r)},z^*)=x^{(r)}+\alpha\cdot z^*$, with $\alpha>0$ maximal.
    If $\alpha\ge n$, then we use the fact that $z^\circ/n$ is feasible, $c^\top z^\circ/n<0$ and $z^*$ is optimal to conclude that $g(x^{(r+1)})\le g(y^\circ)$.
    Applying \Cref{gap-on-N} yields 
    \begin{align*}
        g(y^\circ) &\leq 8\hat g \|y^\circ_B\|_1 + 2\|y^\circ_N\|_1 \\ 
        &\leq 16 n \hat g + 4 \|y^{(p+1)}_N\|_1 \tag{by \Cref{consistency-in-B} and \Cref{increase-on-small-N}}\\
        &\leq 16n\hat g + 4 \|\mx_N(f_3(n,p+1)\cdot \hat g)\|_1 \tag{by \Cref{lem:gap}} \\
        &\leq 16n\hat g + 16n \cdot f_3(n,p+1)\cdot \hat g \tag{$[\tilde g, 1]$ is $\frac14$-polarized and $\mx(1) = \one$} \\
        &\leq f_1(n)\hat g.
    \end{align*}

    On the other hand, if $\alpha<n$, then for $i\in S_p^{(r)}$, we have 
    $x^{(r+1)}_i\leq 3ny^{(p)}_i + 2ny^{(p+1)}_i \leq 3ny^{(p+1)}_i$
    from the constraints in \eqref{sys:long_steps_lp}.
    So, $S_{p+1}^{(r+1)}\supseteq S_p^{(r)}$.
    Similarly, it holds for $i\in B$ that $x^{(r+1)}_i\ge x_i^{(r)}-n x_i^{(r)}/f_4(n)= x_i^{(r)}/2$.
    Thus, the augmentation step must be stopped by a variable $i^*\in N\setminus S_p^{(r)}$ hitting zero.
    Clearly, $i^*\in S_{p+1}^{(r+1)}\setminus S_p^{(r)}$, and $g(x^{(r+1)})<g(x^{(r)})$.
\end{proof}

\subsection{Bounding the Number of Augmentation Steps}
\label{sec:iterations}
In order to bound the number of augmentation steps carried out by Algorithm~\ref{alg:non-det}, it remains to cross the singular value breakpoint $1/\sigma_{\dim(\pi_N(W))}$ in $\mathrm{poly}(n)$ iterations.
We achieve this by showing that the smallest singular value of $\ell^W_N$ can never be too big.
\begin{restatable}{lemma}{smallestsingularvalue}\label{smallest_singular_value}
    Let $\tilde g\le1/8n$.
    Then, $\sigma_{\dim(\pi_N(W))}(\ell_N^W)\le 2n^{3/2}$.
\end{restatable}

\begin{proof}
    We can assume without loss of generality that $N$ is non-empty (otherwise, there is no positive singular value and the claim holds immediately).
    By definition of the smallest singular value, we have $\sigma_{\dim(\pi_N(W))}(\ell_N^W) = \min_{x\in\pi_N(W)\setminus\{\zero\}}\|\ell_N^W(x)\|/\|x\|$.
    Choose $z:=\tilde x^\mcp(\tilde g)-\tilde x^\mcp(1)$.
    By construction, $z\in W$, and $z_N\in\pi_N(W)$.
    By \eqref{approx-N} we have that 
    \begin{equation*}
        \|z_N\|=\|\tilde x^\mcp(1)_N-\tilde x^\mcp(\tilde g)_N\|\ge\|(1/n-4\tilde g)\cdot\mathbf{1}_N\|\ge\|1/(2n)\cdot\mathbf{1}_N\|=\sqrt{|N|}/2n.
    \end{equation*}
    On the other hand, $\|\ell_N^W(z_N)\|\le\|z_B\|\le\sqrt{|B|}$.
    Therefore, $\sigma_{\dim(\pi_N(W))}\le2n\sqrt{|B|/|N|}\le 2n^{3/2}$.
\end{proof}

We now have all the ingredients to prove the main result of this section.
The bound we obtain here is a factor of $O(n)$ weaker than \Cref{main}; the stronger bound will follow from \Cref{thm:amortized} in \Cref{sec:amortization}.

\begin{theorem}\label{thm:main_weak}
Algorithm~\ref{alg:non-det} terminates in 
\[O\left(n^3\log n \sum_{i\in[n]}\slc_{1/2}(x_i^\mcp,g(x^{(0)}))\right)\]
\textsc{Ratio-Circuit} augmentation steps.
\end{theorem}

\begin{proof}
Let $x^{(0)}$ denote the initial feasible point given as input and define $g^{(0)}:=g(x^{(0)})$.
Let $g\in[0,g^{(0)}]$ and let $[g_j,g_{j+1}]$ be the polarized interval containing $g$, with associated polarized partition $(B,N)$.
We define a potential function $\Phi:\R_{\ge0}\to\Z_{\geq 0}$ for the gap of the iterates of the circuit walk as
\begin{equation*}
    \Phi(g) = \begin{cases}
        1+n\cdot\max\{j\in \Z_{\geq 0}:g> g_j\}+C_\sigma(\ell_N^{W_j},[g_{j+1}/g,\infty)), &\text{ if }g>0\\
        0, &\text{ if }g = 0.
    \end{cases}
\end{equation*}

It is easy to see that $\Phi(g(x))=0$ if and only if $g(x)=0$, i.e., $x$ is optimal.
Further, $\Phi(g(x))$ is monotonically nonincreasing as $g(x)$ decreases.
Finally, $\Phi(g)\le O(n\cdot\sum_{i\in[n]}\slc_{1/2}(x_i^\mcp,g^{(0)}))$.
Thus, it remains to show that Algorithm~\ref{alg:non-det} decreases the potential function within $O(n^2\log n)$ circuit augmentation steps.

Let $x\in P$ be a feasible point produced by Algorithm~\ref{alg:non-det}.
Let $[g_j,g_{j+1}]$ be the polarized interval containing $g(x)$, let $\sigma:=\sigma(\ell_N^{W_j})$ and define $k:=C_\sigma(\ell_N^{W_j},[g_{j+1}/g(x),\infty))$.
As long as $k=\dim(\pi_N(W_j))$, Algorithm~\ref{alg:non-det} only uses steps computed by \textsc{Ratio-Circuit}$(A,c,\zero,1/x)$.
By \Cref{smallest_singular_value} and \Cref{iterated-geometric-progress}, the potential function decreases within $O(n\log n)$ augmentation steps.
Thus, we can assume that $k<\dim(\pi_N(W_j))$ and $\sigma_{k+1}>0$.
While $g(x)>1/(f_2(n)\cdot\sigma_{k+1})$, Algorithm~\ref{alg:non-det} again only uses steps computed by \textsc{Ratio-Circuit}$(A,c,\zero,1/x)$.
The same analysis reveals that the algorithm produces a new iterate within $O(n\log n)$ augmentation steps such that $g(x)\le1/(f_2(n)\cdot\sigma_{k+1})$.

If $g(x)>f_1(n)\cdot\hat g$, then we initialize $p=0$ and augment in the direction of a solution to~\eqref{sys:long_steps_lp}.
As long as the condition $g(x)>f_1(n)\cdot\hat g$ is satisfied, we continue running Algorithm~\ref{alg:long-steps}.
By \Cref{scaling-up}, we know that $g(x)$ decreases below $f_1(n)\cdot\hat g$, or we increase the set of small coordinates.
Thus, after at most $|N|$ such augmentation steps, the iterate fulfills $x_i\le f_1(n)\cdot \hat g$ for all $i\in N$. 
Since $x_i\leq 1$ for all $i\in B$, we have $g(x)\le 2nf_1(n)\cdot \hat g $ by \Cref{gap-on-N}.
If the potential function did not decrease, Algorithm~\ref{alg:non-det} again only uses steps of type \textsc{Ratio-Circuit}$(A,c,\zero,1/x)$.
By \Cref{iterated-geometric-progress}, the potential function decreases within at most $O(n^2\log n)$ steps.
Clearly, the total number of iterations is dominated by this last application of \Cref{iterated-geometric-progress}.
\end{proof}

\section{Implementing Algorithm~\ref{alg:non-det}}\label{sec:implementation}

In this section, we remove the max central path and its polarized decomposition from the input, thereby providing a full implementation of Algorithm~\ref{alg:non-det}.
Recall that this additional information is only used to traverse long polarized intervals, as short polarized intervals can be crossed in $O(n^2 \log n)$ iterations of \textsc{RatioCircuit}$(A,c,\zero,1/x)$.
Two major difficulties arise when we try to run \textsc{Long-Steps}:
\begin{enumerate}[label={(\arabic*)}]
    \item Identifying the polarized partition of the current polarized interval;
    \item Computing the points $\bar x$ and $y^{(0)}$.
\end{enumerate} 

To solve these issues, we first show how to approximate the max central path (\Cref{sec:approx_mcp}).
In particular, given any feasible solution $x$ to \eqref{lp} with optimality gap $g$, we approximate $\mx(g)$ to within an $O(n)$-factor using the \textsc{Ratio-Circuit} oracle.
We remark that this is the only place where the dual optimal solution returned by the oracle is used.
After that, we use this approximation to guess the polarized partition (\Cref{sec:guess_partition}), and to compute the points $\bar x$ and $y^{(0)}$ (\Cref{sec:compute_anchors}).
Finally, we put everything together in \Cref{sec:finale}.

\subsection{Approximating the Max Central Path}
\label{sec:approx_mcp}

Given a feasible solution $x$ to \eqref{lp} with optimality gap $g\in \R_{>0}$, we show how to approximate $\mx(g)$ to within a factor of $O(n)$.
First, we approximate the value of $g$, which is unknown to us.

\begin{lemma}\label{lem:gap_approx}
Let $x$ be a feasible solution to \eqref{lp} with optimality gap $g\in \R_{\geq 0}$.
Let $(z,s,\lambda)$ be the primal-dual solution returned by \textsc{Ratio-Circuit}$(A,c,\zero,1/x)$.
Then, $g/n\leq \lambda \leq g$. 
\end{lemma}

\begin{proof}
We have $-g\leq \pr{c}{z} \leq  -g/n$.
The first inequality is due to the feasibility of $x + z$ to \eqref{lp}.
The second inequality holds because for any optimal solution $x^*$ to \eqref{lp}, $(x^* - x)/n$ is feasible to \eqref{sys:min_ratio} with $u := c$, $v := \zero$ and $w := 1/x$.
The proof is complete by noticing that $\pr{c}{z} = -\lambda$.
\end{proof}

Let $(z,s,\lambda)$ be as in \Cref{lem:gap_approx}.
We may assume that $\lambda>0$, as otherwise $g = 0$ and so $x$ is optimal.
Fix a coordinate $i\in [n]$.
To approximate $\mx_i(g)$, we run \textsc{Ratio-Circuit}$(A,-e_i,s/\lambda,2/x-s/\lambda)$.
By the constraint $\zero \leq s \leq \lambda/x$ in \eqref{sys:min_ratio_dual}, we have
\begin{equation}\label{eq:valid_input}
    \frac{s}{\lambda} \geq 0 \qquad \text{and} \qquad \frac2x - \frac{s}{\lambda} \geq \frac1x \geq 0
\end{equation} 
so this is a valid input to \textsc{Ratio-Circuit}.
The next theorem shows that the returned solution yields an $O(n)$-approximation of $\mx_i(g)$.

\begin{theorem}\label{thm:mcp_approx}
Let $x$ be a feasible solution to \eqref{lp} with optimality gap $g\in \R_{> 0}$.
Let $(z,s,\lambda)$ be the primal-dual solution returned by \textsc{Ratio-Circuit}$(A,c,\zero,1/x)$.
Let $z'$ be the primal solution returned by \textsc{Ratio-Circuit}$(A,-e_i, s/\lambda, 2/x - s/\lambda)$.
Then, $x + z'$ is a feasible solution to \eqref{lp} with optimality gap at most $2g$ and 
\[\frac{\mx_i(g)}{2n}\leq x_i + z'_i \leq 2\mx_i(g).\]
\end{theorem}

\begin{proof}
The system \eqref{sys:min_ratio} with $u := -e_i$, $v := s/\lambda$ and $w := 2/x - s/\lambda$ can be equivalently written as
\begin{equation}
    \label{sys:mcp_approx}
    \begin{aligned}
        &\max\; z_i \\
        &\mathrm{s.t.} \quad Az =\zero\, , \\
        &\qquad\; \frac{\pr{c}{z}}{\lambda} + \pr{\frac2x}{z^-} \leq 1\, .
    \end{aligned}
\end{equation}
Notice that $\pr{c}{z} = \pr{s}{z}$ for all $z\in \ker(A)$ because $s \in \im(A) + c$.

First, $x+z'$ is feasible to \eqref{lp} because $z'\in \ker(A)$ and
\[\pr{\frac1x}{(z')^-} \leq \pr{\frac{s}{\lambda}}{(z')^+} + \pr{\frac2x - \frac{s}{\lambda}}{(z')^-} \leq 1\]
by \eqref{eq:valid_input}.
Next, letting $s^*$ be any optimal solution to \eqref{dlp}, $x + z'$ has optimality gap
\[\pr{s^*}{x+z'} = g + \pr{c}{z'} \leq g + \lambda \leq 2g. \]
Hence, 
\[x_i + z'_i \leq \mx_i(2g)\leq 2\mx_i(g)\]
where the first inequality is by the monotonicity of $\mx_i$, while the second inequality follows from the concavity and nonnegativity of $\mx_i$.

It is left to show the lower bound.
Let $\hat x$ be a feasible solution to \eqref{lp} such that $\hat x_i = \mx_i(g)$.
Since $(\hat x - x)/(2n)$ is feasible to \eqref{sys:mcp_approx}, we have $z'_i \geq (\mx_i(g) - x_i)/(2n)$ by the optimality of $z'$ to \eqref{sys:mcp_approx}.
Thus,
\[x_i + z'_i \geq \frac{\mx_i(g)}{2n} + \frac{2n-1}{2n} x_i \geq \frac{\mx_i(g)}{2n}. \qedhere\]
\end{proof}

\subsection{Guessing the Polarized Partition}
\label{sec:guess_partition}

In this subsection, we show how to guess the polarized partition $(B,N)$ of the current polarized interval, when the interval is sufficient long.  
The main idea is to keep track of an approximation of the max central path over a sequence of calls to \textsc{Ratio-Circuit} with Wallacher's rule. 
At the end of this sequence, the variables that did not change much (compared to the start) will form our guess of $B$, while the rest will form our guess of $N$.
The number of such calls will be $O(n\log n)$, which can be easily paid for by the overall running time analysis.

We first set up the required notation.
For $t\geq 0$, let $x^{(t)}$ be the $t$-th iterate of our circuit walk. 
Let $g^{(t)}$ denote the optimality gap of $x^{(t)}$.
By applying \Cref{lem:gap_approx} to $x^{(t)}$, we can approximate $g^{(t)}$.
In particular, we obtain $(z^{(t)},s^{(t)},\lambda^{(t)})$ such that 
\begin{equation}\label{gap_approx}
    \frac{g^{(t)}}{n}\leq \lambda^{(t)}\leq g^{(t)}.
\end{equation}
By applying \Cref{thm:mcp_approx} to $x^{(t)}$, we can also approximate $\mx(g^{(t)})$.
In particular, for every $i\in [n]$, we obtain $z^{(t,i)}$ such that $x^{(t,i)} := x^{(t)} + z^{(t,i)}$ is feasible to \eqref{lp} with optimality gap at most $2g^{(t)}$.
Moreover, denoting $\hat x^{(t)}_i := x^{(t,i)}_i$ for all $i\in [n]$, we have
\begin{equation}\label{mcp_approx}
    \frac{\mx(g^{(t)})}{2n} \leq \hat x^{(t)} \leq 2\mx(g^{(t)}).
\end{equation}
Let us define $\bar x^{(t)} := \sum_{i=1}^n  x^{(t,i)}/n$ as a feasible approximation of $\hat x^{(t)}$.
Clearly, $\bar x^{(t)}$ is feasible to \eqref{lp} with optimality gap at most $2g^{(t)}$.
Furthermore,
\begin{equation}\label{mcp_approx_feas}
    \frac{\mx(g^{(t)})}{2n^2}\leq \bar x^{(t)} \leq 2\mx(g^{(t)}).
\end{equation}

Now, fix an iteration $t_0$.
We know that $g^{(t_0)}$ lies in some polarized interval $[g_j, g_{j+1}]$, which is unknown to us.
To guess the polarized partition $(B,N)$ of $[g_j, g_{j+1}]$, let $t\geq t_0$ be an iteration after more than $\lceil 4n\log(4n)\rceil$ calls to \textsc{Ratio-Circuit}$(A,c,\zero,1/x)$ starting from iteration $t_0$.
We define our guess as
\begin{equation}
    \tilde B := \left\{ i\in [n]:\frac{1}{16n}\leq \frac{\hat x_i^{(t)}}{\hat x_i^{(t_0)}}\leq 4n \right\} \qquad \qquad \tilde N := \left\{i\in [n]: \frac{1}{4n^2} \cdot \frac{\lambda^{(t)}}{\lambda^{(t_0)}}\leq \frac{\hat x_i^{(t)}}{\hat x_i^{(t_0)}}\leq 16n^2 \cdot \frac{\lambda^{(t)}}{\lambda^{(t_0)}} \right\}.
\end{equation}

Like $[g_j, g_{j+1}]$, the subinterval $[g_j, g^{(t_0)}]$ is $\frac14$-polarized and has the same polarized partition $(B,N)$.
This is because \eqref{eq:polarized-gamma} holds when $g_{j+1}$ is replaced with $g^{(t_0)}$.
To simplify the analysis, we will assume that $g^{(t_0)} = 1$ and $\mx_i(g^{(t_0)}) = 1$ for all $i\in [n]$.
We also denote $\tilde g := g_j$.

The following proposition shows that our guess is correct whenever $g^{(t)}$ lies in the same polarized interval as $g^{(t_0)}$.

\begin{proposition}\label{guess_partition}
If $g^{(t)}\in [\tilde g, 1]$, then $\tilde B = B$ and $\tilde N = N$.
\end{proposition}

\begin{proof}
By \eqref{mcp_approx}, we have 
\begin{align*}
    \qquad \frac{1}{2n} &\leq \hat x^{(t_0)}_i \leq 2   &\forall i\in [n].
\end{align*}
Since $[\tilde g, 1]$ is $\frac14$-polarized, by \eqref{gap_approx} and \eqref{mcp_approx} we also have
\begin{align*}
    \frac{1}{8n}&\le\hat x^{(t)}_i\le 2 &\forall i\in B,\\
    \frac{\lambda^{(t)}}{n\lambda^{(t_0)}}\cdot \frac{1}{2n}&\le\hat x^{(t)}_i\le 8 \cdot \frac{n\lambda^{(t)}}{\lambda^{(t_0)}} &\forall i\in N.
\end{align*}
Hence, $B\subseteq \tilde B$ and $N\subseteq \tilde N$.
To show equality, it suffices to show that $\tilde B\cap \tilde N = \emptyset$.
By \Cref{iterated-geometric-progress}, 
\begin{equation}\label{eq:gap_difference}
  \lambda^{(t)} \leq g^{(t)} < \frac{1}{16^2n^4} \leq \frac{\lambda^{(t_0)}}{16^2n^3}, 
\end{equation}
which completes the proof upon rearranging terms.
\end{proof}

\subsection{Computing the Points $\bar x$ and $y^{(0)}$}
\label{sec:compute_anchors}

Equipped with a candidate partition $(\tilde B, \tilde N)$ for the polarized interval $[g_j, g_{j+1}]$, the next step is to compute the lifting operator $\ell^{W_j}_{\tilde N}:\pi_{\tilde N}(W_j)\to \pi_{\tilde B}(W_j)$, where $W_j := {\rm diag}(\mx(g_{j+1}))^{-1} W$.
This seems difficult because we don't know $\mx(g_{j+1})$.
Fortunately, we have an approximation of the max central path at $g^{(t_0)}$. 
So, let $\ell^{W'}_{\tilde N}:\pi_{\tilde N}(W')\to \pi_{\tilde B}(W')$ where $W' := {\rm diag}(\mx(g^{(t_0)}))^{-1} W$.
We compute the lifting operator $\ell^{\tilde W}_{\tilde N}:\pi_{\tilde N}(\tilde W)\to \pi_{\tilde B}(\tilde W)$, where $\tilde W := {\rm diag}(\hat x^{(t_0)})^{-1} W$.
Letting $M$ be a matrix whose columns form an orthonormal basis of $\tilde W$, it can computed as
\[\ell^{\tilde W}_{\tilde N} = M_{\tilde B,\bullet}M^\dagger_{\tilde N,\bullet}.\]
See Lemma 2.3 of \cite{DHNV24} for a proof of the formula above.

Let $\sigma:= \sigma(\ell^{W'}_{\tilde N})$ and $\tilde \sigma := \sigma(\ell^{\tilde W}_{\tilde N})$ be the vector of sorted singular values of the two lifting operators.
Assuming that $\mx(g^{(t_0)})>\zero$, both vectors contain $\dim(\pi_{\tilde N}(W))$ entries.
By \eqref{mcp_approx} and \Cref{lemma:singular-value-scaling}, we have
\begin{equation}\label{eq:singular_values}
    \frac{1}{4n} \cdot\tilde \sigma\leq \sigma\leq 4n\cdot\tilde\sigma.
\end{equation}
So, the singular values of $\ell^{W'}_{\tilde N}$ and $\ell^{\tilde W}_{\tilde N}$ are multiplicatively close.

Let $k$ be the largest integer such that $4n^2\lambda^{(t_0)}/\tilde \sigma_k < \lambda^{(t)}$ (we set $k=0$ if no such integer exists).
For simplicity, we will again assume that $g^{(t_0)} = 1$ and $\mx_i(g^{(t_0)}) = 1$ for all $i\in [n]$.
We also denote $\tilde g := g_j$.

\begin{proposition}
If $g^{(t)}\in [\tilde g, 1]$, then $k<\dim(\pi_{\tilde N}(W))$.
\end{proposition}

\begin{proof}
By \eqref{eq:gap_difference}, we have $\tilde g\leq g^{(t)}\leq 1/(8n)$. So, we can apply \Cref{smallest_singular_value} and \eqref{eq:singular_values} to obtain $\tilde \sigma_{\dim(\pi_{\tilde N}(W))}\leq 4n\sigma_{\dim(\pi_{\tilde N}(W))} \leq 8n^{5/2}$.
Then, 
\[\frac{4n^2\lambda^{(t_0)}}{\tilde \sigma_{\dim(\pi_{\tilde N}(W))}}\geq \frac{4n^2(1/n)}{8n^{5/2}} = \frac{1}{2n^{3/2}} >\lambda^{(t)}\]
where the last inequality is due to \eqref{eq:gap_difference}.
\end{proof}

If $k\geq 1$, then by \eqref{eq:singular_values} and \eqref{gap_approx}
\begin{equation}\label{eq:left_endpoint}
    \frac{1}{\sigma_k} \leq \frac{4n}{\tilde \sigma_k} \leq \frac{4n^2\lambda^{(t_0)}}{\tilde \sigma_k} < \lambda^{(t)} \leq g^{(t)}.
\end{equation}
Let $\hat g := \max\{1/\sigma_k,\tilde g\}$ (we set $\hat g := \tilde g$ if $k=0$).
Our goal is to reduce $g^{(t)}$ to below $\hat g$.
Due to our choice of $k$,
\[g^{(t)} \leq n\lambda^{(t)} \leq \frac{4n^3\lambda^{(t_0)}}{\tilde \sigma_{k+1}} \leq \frac{16n^4}{\sigma_{k+1}}.\]
We may assume that
\begin{equation}\label{eq:right_distance}
    g^{(t)}\leq \frac{1}{512n^{7/2}\sigma_{k+1}}.
\end{equation}
by performing additional $\lceil n\log(8192n^{15/2}) \rceil$ calls to \textsc{Ratio-Circuit}$(A,c,\zero,1/x)$.

The hard case is when $g^{(t)}>f_1(n)\cdot\hat g$.
In order to apply \textsc{Long-Steps}, we need to produce a feasible point $\bar x$ satisfying Properties~\ref{cond:x_gap} and \ref{cond:x_lbound}, and another feasible point $y^{(0)}$ satisfying Properties \ref{cond:y_gap} and \ref{cond:y_lbound}.

We set $\bar x := \bar x^{(t)}$.
Then, Property~\ref{cond:x_gap} holds because $g(\bar x)\leq 2g^{(t)}$, while
Property~\ref{cond:x_lbound} is satisfied by \eqref{mcp_approx_feas}.

To produce $y^{(0)}$, let $V_d$ be the singular subspace for $\ell^{\tilde W}_{\tilde N}$ of dimension $d := \dim(\pi_{\tilde N}(W)) - k$, and let $\Pi_{V_d}$ be the orthogonal projection onto $V_d$.
We project $-\bar x^{(t)}_{\tilde N}/\hat x_{\tilde N}^{(t_0)}$ onto $V_d$, apply the lifting operator $\ell^{\tilde W}_{\tilde N}$, and rescale by $\hat x^{(t_0)}$ to obtain the following vector in $W$
\[\Delta x := \hat x^{(t_0)} \circ \left(\ell^{\tilde W}_{\tilde N}\bigl(\Pi_{V_d}(-\bar x^{(t)}_{\tilde N}/\hat x_{\tilde N}^{(t_0)})\bigr),\Pi_{V_d}(-\bar x^{(t)}_{\tilde N}/\hat x_{\tilde N}^{(t_0)})\right).\]
The following lemma bounds the distortion caused by the projection on $\tilde N$.

\begin{lemma}\label{lem:proj_error}
If $g^{(t)}\in [\tilde g, 1]$, then $\|\Delta x_{\tilde N} + \bar x^{(t)}_{\tilde N}\|\leq 48 n^{5/2}\hat g$.
\end{lemma}

\begin{proof}
Since $g^{(t)}\in [\tilde g, 1]$, we have $\tilde B = B$ and $\tilde N = N$.
Let us denote $z := (\tilde x^\mcp(\tilde g) - \bar x^{(t)})/\hat x^{(t_0)}\in \tilde W$.
Then,
\begin{align*}
    \left\|\Delta x_N+\bar x^{(t)}_{N} \right\| &=  \left\|\hat x_{N}^{(t_0)} \circ \Pi_{V_d}(-\bar x^{(t)}_{N}/\hat x_{N}^{(t_0)}) + \bar x^{(t)}_{N} \right\| \\
    &\leq \left\|\hat x^{(t_0)}_N\right\|_\infty  \left\|\Pi_{V_d}(-\bar x^{(t)}_{N}/\hat x_{N}^{(t_0)}) + \bar x^{(t)}_{N}/\hat x_{N}^{(t_0)} \right\| \\ 
    &\leq 2  \left\|\Pi_{V_d}(-\bar x^{(t)}_{N}/\hat x_{N}^{(t_0)}) + \bar x^{(t)}_{N}/\hat x_{N}^{(t_0)} \right\| \tag{\eqref{mcp_approx} and $\mx(g^{(t_0)}) = \one$}\\
    &\leq 2 \left\| \Pi_{V_d}(z_N) +  \bar x^{(t)}_{N}/\hat x_{N}^{(t_0)}  \right\|  \tag{orthogonal projection is closest} \\
    &= 2 \left\| \Pi_{V_d}(z_N) - z_N + \tilde x^\mcp_N(\tilde g)/\hat x_{N}^{(t_0)} \right\| \\
    &\leq 2 \left(\left\| \Pi_{V_d}(z_N)- z_N \right\|  + \left\|\tilde x^\mcp_N(\tilde g)/\hat x_{N}^{(t_0)}\right\| \right)\\
    &\leq 2 \left(\frac{\|z_B\|}{\tilde \sigma_k} + 8n^{3/2}\tilde g \right).  \tag{\Cref{projections}, \eqref{approx-N} and \eqref{mcp_approx}} 
\end{align*}
Since $[\tilde g, 1]$ is $\frac14$-polarized and $\mx(1) = \one$, for every $i\in B$, we have
$\tilde x^\mcp_i(\tilde g)\leq 1$ by \eqref{approx-B} and $\bar x_i^{(t)} \leq 2$ by \eqref{mcp_approx_feas}.
Hence, $|z_i| \leq 4n$ for all $i\in B$ by \eqref{mcp_approx} and
\[\left\| \Delta x_N + \bar x^{(t)}_{N} \right\| \leq  \frac{8n^{3/2}}{\tilde \sigma_k} + 16n^{3/2}\tilde g \leq  \frac{32n^{5/2}}{\sigma_k} + 16n^{3/2}\tilde g \leq 48 n^{5/2}\hat g,\]
where the second inequality is by \eqref{eq:singular_values}.
\end{proof}

By \Cref{lem:proj_error}, if $\|\Delta x_{\tilde N} + \bar x^{(t)}_{\tilde N}\|\geq 48n^{7/2}\lambda^{(t)}/f_1(n)$, then $\lambda^{(t)}\leq f_1(n)\hat g/n$.
Combining this with \eqref{gap_approx} yields $g^{(t)}\leq f_1(n) \hat g$.
In this case, we can cross $\hat g$ in  $\lceil n\log(f_1(n))\rceil = O(n^2\log n)$ iterations of \textsc{Ratio-Circuit}$(A,c,\zero,1/x)$.
Henceforth, we assume that
\begin{equation}\label{eq:proj_error}
    \left\|\Delta x_{\tilde N} + \bar x^{(t)}_{\tilde N} \right\|< \frac{48n^{7/2}\lambda^{(t)}}{f_1(n)}.
\end{equation}

The next lemma bounds the norm of $\Delta x_B$.

\begin{lemma}\label{lem:lifting_cost}
If $g^{(t)}\in [\tilde g, 1]$ and \eqref{eq:right_distance} holds, then $\|\Delta x_B\|\leq 1/(16n^2)$.
\end{lemma}

\begin{proof}
Since $g^{(t)}\in [\tilde g, 1]$, we have $\tilde B = B$ and $\tilde N = N$.
As $\Delta x_B/\hat x^{(t_0)}_B = \ell^{\tilde W}_N(\Delta x_N/\hat x^{(t_0)}_N)$ and $\Delta x_N/\hat x^{(t_0)}_N\in V_d$, we obtain
\begin{align*}
    \left\|\Delta x_B/\hat x^{(t_0)}_B\right\| &\leq \tilde \sigma_{k+1} \left\|\Delta x_N/\hat x^{(t_0)}_N\right\| \\ 
    &\leq \tilde \sigma_{k+1} \left\|\bar x^{(t)}_{N}/\hat x^{(t_0)}_N \right\| \tag{$\Delta x_N/\hat x^{(t_0)}_N = \Pi_{V_d}(-\bar x^{(t)}_N/\hat x^{(t_0)}_N)$}\\
    &\leq \tilde \sigma_{k+1} \left\|2n\bar x^{(t)}_{N} \right\|  \tag{by \eqref{mcp_approx} and $\mx(g^{(t_0)}) = \one$} \\
    &\leq \tilde \sigma_{k+1} \left\|4n\mx_N(g^{(t)}) \right\| \tag{by \eqref{mcp_approx_feas}}\\
    &\leq \tilde \sigma_{k+1}\cdot 16n^{3/2}g^{(t)} \tag{$[\tilde g, 1]$ is $\frac14$-polarized} \\
    &\leq \frac{1}{32n^2} \tag{by \eqref{eq:right_distance}}.
\end{align*}
By \eqref{mcp_approx} and $\mx(g^{(t_0)}) = \one$ again, we get
\[\left\|\Delta x_B\right\| \leq \left\|\hat x^{(t_0)}_B\right\|_\infty \left\|\Delta x_B/\hat x^{(t_0)}_B\right\| \leq \frac{2}{32n^2} = \frac{1}{16n^2} \qedhere\]
\end{proof}

Let $\alpha>0$ be the maximal step size such that
\[x':=\bar x^{(t)} + \alpha  \Delta x\] 
satisfies $x'_i \geq 0$ for all $i\in \tilde B$ and $x'_i\geq 8n^2 \hat x^{(t_0)}_i/\tilde \sigma_k$ for all $i\in \tilde N$.
We claim that $x'$ is the desired point $y^{(0)}$.
First, observe that for every $i\in N$, we have
\[x'_i\geq \frac{8n^2 \hat x^{(t_0)}_i}{\tilde \sigma_k} \geq \frac{4n\mx_i(g^{(t_0)})}{\tilde \sigma_k} = \frac{4n}{\tilde \sigma_k} \geq  \frac{1}{\sigma_k}.\]
The second and last inequalities follow from \eqref{mcp_approx} and \eqref{eq:singular_values} respectively, while the equality is due to $\mx(g^{(t_0)}) = \one$.
So, it is left to show that the optimality gap of $x'$ is at most $f_6(n)\cdot \hat g$.

\begin{lemma}
If $g^{(t)}\in [\tilde g, 1]$ and \eqref{eq:right_distance},\eqref{eq:proj_error} hold, then $g(x')\leq f_6(n) \cdot \hat g$.
\end{lemma}

\begin{proof}
For every $i\in N$, we have
\[\bar x^{(t)}_i + \Delta x_i \leq \left\| \bar x^{(t)}_N + \Delta x_N \right\| < \frac{48n^{7/2}\lambda^{(t)}}{f_1(n)} \leq \frac{48n^{7/2}g^{(t)}}{f_1(n)} \leq \frac{48n^{7/2}\mx_i(g^{(t)})}{f_1(n)} \leq \frac{96n^{11/2}\bar x^{(t)}_i}{f_1(n)}.\]
The second inequality is by \eqref{eq:proj_error}, the third inequality is by \eqref{gap_approx}, the fourth inequality is due to \Cref{lem:mcp_bounds} and $\mx(1) = \one$, and the last inequality follows from \eqref{mcp_approx_feas}.
Hence, 
\begin{equation}\label{eq:step_size}
    \alpha \leq \frac{\bar x^{(t)}_i}{-\Delta x_i}\leq \left(1-\frac{96n^{11/2}}{f_1(n)}\right)^{-1} < 2.
\end{equation}
Since $\bar x^{(t)}_j \geq 1/(8n^2)$ for all $j\in B$ by \eqref{mcp_approx_feas} and the fact that $[\tilde g, 1]$ is $\frac14$-polarized, it follows from \Cref{lem:lifting_cost} that the step size is limited by some coordinate $i^*\in N$.

From the definition of $x'$, we obtain
\begin{align*}
\frac{64n^3}{\sigma_k} \geq \frac{16n^2}{\tilde \sigma_k} = \frac{16n^2\mx_{i^*}(g^{(t_0)})}{\tilde \sigma_k} &\geq \frac{8n^2\hat x^{(t_0)}_{i^*}}{\tilde \sigma_k} \tag{by \eqref{eq:singular_values} and \eqref{mcp_approx}}\\
&= x'_{i^*} \\
&= (1-\alpha)\bar x^{(t)}_{i^*} + \alpha (\bar x^{(t)}_{i^*} + \Delta x_{i^*}) \\ 
&\geq  (1-\alpha)\bar x^{(t)}_{i^*} - 48\alpha n^{5/2}\hat g, \tag{by \Cref{lem:proj_error}}
\end{align*}
which implies that $\alpha \geq (\bar x^{(t)}_{i^*}-64n^3/\sigma_k)/(\bar x^{(t)}_{i^*} + 48 n^{5/2}\hat g)$.
Hence, for every $i\in N$, we have
\begin{align*}
    x'_i &= (1-\alpha)\bar x^{(t)}_{i} + \alpha (\bar x^{(t)}_{i} + \Delta x_{i}) \\ 
    &\leq \frac{48 n^{5/2}\hat g + 64n^3/\sigma_k}{\bar x^{(t)}_{i^*} + 48 n^{5/2}\hat g} \cdot \bar x^{(t)}_i + 96 n^{5/2}\hat g  \tag{by \eqref{eq:step_size} and \Cref{lem:proj_error}} \\
    &\leq \frac{112 n^3\hat g}{\bar x^{(t)}_{i^*} } \cdot \bar x^{(t)}_i + 96 n^{5/2}\hat g\\
    &\leq  \frac{112 n^3\hat g}{\mx_{i^*}(g^{(t)})/(2n^2)} \cdot 2\mx_{i}(g^{(t)}) + 96 n^{5/2}\hat g  \tag{by \eqref{mcp_approx_feas}} \\
    &\leq  \frac{112 n^3\hat g}{g^{(t)}/(2n^2)} \cdot 8g^{(t)} + 96 n^{5/2}\hat g  \tag{$[\tilde g, 1]$ is $\frac14$-polarized and $\mx(1) = \one$} \\ 
    &\leq 1888n^5 \hat g
\end{align*}

On the other hand, for every $j\in B$, we have $x'_j= \bar x^{(t)}_j + \alpha \Delta x_j < 1 + 2/(16n^2) \leq 9/8$ by \Cref{lem:lifting_cost}.
Therefore, \Cref{gap-on-N} yields
\[g(x') \leq 8\tilde g\|x'_B\|_1 + 2\|x'_N\|_1 \leq 9n \tilde g + 3776n^6\hat g \leq 3785n^6\hat g = f_6(n)\cdot \hat g. \qedhere \]
\end{proof}

\subsection{Putting Everything Together}
\label{sec:finale}

Fix an iterate $x^{(t_0)}$ and let $[g_j, g_{j+1}]$ be the polarized interval containing $g^{(t_0)}$.
Let $x^{(t)}$ be an iterate after more than $\lceil 4n\log(4n)\rceil$ calls to \textsc{Ratio-Circuit}$(A,c,\zero,1/x)$, starting from $x^{(t_0)}$.
We have shown how to compute a guess $(\tilde B, \tilde N)$ of the polarized partition for $[g_j, g_{j+1}]$.
This guess is correct if $g^{(t)}\in [g_j, g_{j+1}]$.
We have also shown how to compute the points $\bar x$ and $y^{(0)}$, and they satisfy Properties\ref{cond:x_gap}, \ref{cond:x_lbound}, \ref{cond:y_gap}, \ref{cond:y_lbound} if $g^{(t)}\in [g_j, g_{j+1}]$ and \eqref{eq:right_distance}, \eqref{eq:proj_error} hold. 
However, it is unclear how to check if $g^{(t)}\in [g_j, g_{j+1}]$ because we don't have access to the polarized decomposition of $\mx$.
For the same reason, we don't know $\hat g$, which is needed by \textsc{Long-Steps}.

Fortunately, we can afford to be oblivious because the optimal value of \eqref{sys:long_steps_lp} is always nonpositive.
In particular, we alternate between $\Theta(n^2\log n)$ `short' steps given by Wallacher's rule, and $n$ `long' steps given by \eqref{sys:long_steps_lp} (with increasing values of $p$).
If $g^{(t)}\in [g_j, g_{j+1}]$ and \eqref{eq:right_distance}, \eqref{eq:proj_error} hold, then by \Cref{scaling-up}, we will achieve $g^{(t+p)}\leq f_1(n)\cdot \hat g$ in $p\leq n$ long steps.
Otherwise, these steps do not increase our optimality gap.

A complete pseudocode is provided in Algorithm~\ref{alg:implementation}.
Instead of using \textsc{Long-Steps}, it uses \textsc{Long-Steps-Forced} (Algorithm~\ref{alg:long-steps-forced}).
The difference between these two subroutines is that the latter does not check whether $f_1(n)\cdot \hat g< g(x)$ at the start of every iteration.
It just runs $n$ iterations blindly, unless $y^{(p+1)}$ does not exist.

\begin{algorithm}[!h]\label{alg:long-steps-forced}
\caption{\textsc{Long-Steps-Forced}}
\SetKwInOut{Input}{Input}
\SetKwInOut{Output}{Output}
\Input{Feasible solution $x\in P$, points $\bar x, y^{(0)}\in P$, subset of coordinates $\tilde N\subseteq [n]$}
\Output{Feasible solution $x\in P$}
Let $L$ be the ray containing $[y^{(0)}, \bar x]$, where $L(g)$ is the point on $L$ with optimality gap $g\geq 0$\;
\For{$p=0$ \KwTo $n-1$}{\label{pc:inner-while}
    $S_p \gets \{i\in \tilde N: x_i \leq 3ny^{(p)}_i\}$\;
    $y^{(p+1)} \gets L(g)$ for $g\geq 0$ minimal such that $L_{\tilde N}(g)\geq 4ny^{(p)}_{\tilde N}$\;
    \If{$y^{(p+1)}$ does not exist}{
        break\;
    }
    Compute elementary vector $z\in\elementary{A}$ as a solution to~\eqref{sys:long_steps_lp}\;
    $x\gets\aug_P(x,z)$\;
}

\Return $x$\;
\end{algorithm}

\begin{algorithm}[!h]\label{alg:implementation}
\caption{Circuit augmentation algorithm}
\SetKwInOut{Input}{Input}
\SetKwInOut{Output}{Output}
\Input{Bounded instance of \eqref{lp} with constraint matrix $A\in\R^{m\times n}$, feasible polyhedron $P\subseteq\R^n$, feasible solution $x^{(0)}\in P$}
\Output{Optimal solution $x^*$}
$x\gets x^{(0)}$\;
\While{$x$ is not optimal}{\label{alg2:main-while}
    $(z^{(0)},s^{(0)},\lambda^{(0)}) \gets \textsc{Ratio-Circuit}(A,c,\zero,1/x)$\;
    \For(\tcc*[f]{Approximate max central path}){$i=1$ \KwTo $n$}{
        $(z',s',\lambda') \gets \textsc{Ratio-Circuit}(A,-e_i,s^{(0)}/\lambda^{(0)},2/x-s^{(0)}/\lambda^{(0)})$\;
        $\hat x^{(0)}_i\gets x_i + z'_i$\;
    }
    \For(\tcc*[f]{Short steps}){$t=1$ \KwTo $\lceil n\log(4f_1(n)) \rceil$}{\label{alg2:slow-for}
    $(z,s,\lambda) \gets \textsc{Ratio-Circuit}(A,c,\zero,1/x)$\;
    $x\gets\aug_P(x,z)$\;}
    $(z^{(1)},s^{(1)},\lambda^{(1)}) \gets \textsc{Ratio-Circuit}(A,c,\zero,1/x)$\;\label{alg2:basepoint}
    \For(\tcc*[f]{Approximate max central path}){$i=1$ \KwTo $n$}{
        $(z',s',\lambda') \gets \textsc{Ratio-Circuit}(A,-e_i,s^{(1)}/\lambda^{(1)},2/x-s^{(1)}/\lambda^{(1)})$\;
        $\hat x^{(1)}_i\gets x_i + z'_i$\; $x^{(1,i)}\gets x + z'$\;
    }
    $\tilde B\gets\{i\in[n]:1/(16n)\le \hat x^{(1)}_i/\hat x_i^{(0)}\le4n\}$ \tcc*[r]{Guess polarized partition} 
    $\tilde N\gets\{i\in[n]:\lambda^{(1)}/(4n^2\lambda^{(0)})\le \hat x^{(1)}_i/\hat x_i^{(0)}\le 16n^2\lambda^{(1)}/\lambda^{(0)}\}$\;
    \If{$\tilde B$ and $\tilde N$ partition $[n]$}{
        Compute the lifting operator $\ell_{\tilde N}^{\tilde W}$ where $\tilde W :=\mathrm{diag}(\hat x^{(0)})^{-1}W$ and $W:=\ker(A)$\;\label{alg2:lifting}
        Compute singular values $\tilde \sigma = \sigma(\ell_{\tilde N}^{\tilde W})$\;
        $k\gets\max\{i:4n^2\lambda^{(0)}/\tilde \sigma_i<\lambda^{(1)}\}$\label{alg2:k} \tcc*[r]{Set $k=0$ if no such $i$ exists}
        \If{$k<\dim(\pi_{\tilde N}(W))$}{
        $\bar x\gets\sum_{i=1}^n x^{(1,i)}/n$\;
        Compute singular subspace $V_d$ for $\ell_{\tilde N}^{\tilde W}$ of dimension $d:=\dim(\pi_{\tilde N}(W))-k$\;
        $\Delta x\gets\hat x^{(0)} \circ \left(\ell^{\tilde W}_{\tilde N}\bigl(\Pi_{V_d}(-\bar x_{\tilde N}/\hat x_{\tilde N}^{(0)})\bigr),\Pi_{V_d}(-\bar x_{\tilde N}/\hat x_{\tilde N}^{(0)})\right)$\;
        $y^{(0)}\gets\bar x+\alpha\Delta x$ for $\alpha\geq0$ maximal such that $y^{(0)}_{\tilde B}\ge \zero$ and $y^{(0)}_{\tilde N}\ge8n^2\hat x^{(0)}_{\tilde N}/\tilde \sigma_k$\;
        \If{$y^{(0)}$ exists}{
            $x\gets \textsc{Long-Steps-Forced}(x,\bar x, y^{(0)}, \tilde N)$
        }
        }
    } 
}
\Return $x$\;
\end{algorithm}

\section{An Improved Amortized Iteration Bound}\label{sec:amortization}
In this section, we show that the analysis of the number of iterations in the proof of \Cref{thm:main_weak} is not tight, and we can improve upon it by a factor of $n$.
A similar observation has been used by Allamigeon et al.~\cite{ADL25} to improve the iteration bound of their interior point method.
We remark that an adaption of their proof strategy would imply a similar result for our case, but we choose to provide a shorter new analysis.
The main idea relies on the notion of an \emph{ideal potential function} based on a carefully scaled lifting operator, that scales according to the gap of the iterate. 
With the help of this function, we can show that we do not need to deal with the full range of singular values within each polarized interval.
Further, the singular value decomposition that we use within Algorithm~\ref{alg:non-det} approximates the one used in the ideal potential function, which ensures that it is valid to measure our progress in this way.

We begin by giving a stronger version of \Cref{lem:polarization}, that is, \Cref{thm:polarized-approx}.
In addition to the polarized decomposition of the max central path, we obtain a polynomial factor approximation of the max central path that respects the polarized decomposition and has strong polarization properties itself.
We note that our notion of polarization is defined with respect to the max central path, as opposed to~\cite{ADL25}.
The statement of \Cref{thm:polarized-approx} can be obtained as an implicit corollary of~\cite[Theorem~1.8, Lemma~4.5]{ADL25}, by observing that both proofs use the same polarized decomposition.
We remark that our indexing of the partitions $(B^{(j)},N^{(j)})$ is shifted by one, for consistency with the remainder of our paper. 
Further, we omit one term on the left-hand side of the inequality of \Cref{thm:polarized-approx}, part~\ref{item-bound}, making the claimed bound weaker.

\begin{theorem}[{\cite[Theorem~1.8, Lemma~4.5]{ADL25}}]\label{thm:polarized-approx}
    Let $x\in P$ with optimality gap $g$ and $\eta\in (0,1]$.
    There exist points $0 = g_0 < g_1 < \dots < g_r = g$ with $r \leq 2\sum_{i\in [n]}\slc_{\eta}(\mx_i,g)$, partitions $(B^{(j)},N^{(j)})$ of $[n]$ for each $j\in\{0,\dots,r-1\}$, and piecewise linear functions $h_i:[0,g_r]\to\R_{\ge0}$ for each $i\in[n]$, such that the following conditions hold:
    \begin{enumerate}
        \item each interval $[g_{j}, g_{j+1}]$ is $(\eta/2)$-polarized with associated polarized partition $(B^{(j)},N^{(j)})$ for each $j\in\{1,\dots,r\}$,
        \item $h_i$ is constant whenever $i\in B^{(j)}$, i.e., $h_i(\mu)=h_i(g_{j})$ for all $\mu\in[g_{j},g_{j+1}]$ and $i\in B^{(j)}$,\label{item-B} 
        \item $h_i$ is scaling whenever $i\in N^{(j)}$, i.e., $h_i(\mu)=h_i(g_{j})\cdot \mu/g_{j}$ for all $\mu\in[g_{j},g_{j+1}]$ and $i\in N^{(j)}$,\label{item-N}
        \item $h_i$ is an $\eta/2$-approximation of $x_i^\mcp$, i.e., $x^\mcp_i(\mu)\ge h_i(\mu)\ge\frac\eta2x_i^\mcp(\mu)$ for all $\mu\in[0,g_r]$,\label{item-approx}
        \item $\sum_{i\in[r-1]}|N^{(i-1)}\triangle N^{(i)}|\le 2\sum_{i\in [n]}\slc_{\eta}(\mx_i,g_r)$.\label{item-bound}
    \end{enumerate}  
\end{theorem}

The following lemma describes the behavior of the singular values of the lifting operator under changes of the partition.

\begin{lemma}[{\cite[Lemma~8.2]{ADL25}}]\label{lemma:partition-change}
    Let $W\subseteq\R^n$ be a subspace, and consider two partitions $(B,N)$ and $(\hat B,\hat N)$ of $[n]$.
    Then, the lifting operators $\ell:=\ell_N^W$ and $\hat\ell:=\ell_{\hat N}^W$ satisfy
    \begin{equation*}
        \sigma_i(\ell)\ge\sigma_{i+|N\triangle\hat N|}(\hat \ell),\qquad\forall i\ge1.
    \end{equation*}
\end{lemma}

We are ready to define the \emph{ideal potential function}.
Let $g_0,\dots,g_r$ be a polarized partition with associated polarized partitions $(B^{(j)},N^{(j)})$ and polarized approximation $h_i$ of $x^\mcp_i$ for all $i\in [n]$ as in \Cref{thm:polarized-approx}.
Let $g:=g(x)$ be the gap of the current iterate.
If $g>0$, let $a := \max\{i\in \Z:g_i<g\}$, $W_a:=\mathrm{diag}(h(g_{a+1}))^{-1}W$ and $\ell_a:=\ell_{N^{(a)}}^{W_a}$.
Then, we define $\Phi^{\mathrm{id}}:\R_{\ge0}\to\Z_{\ge0}$ as
\begin{equation*}
    \Phi^{\mathrm{id}}(g) = \begin{cases}
        1+\sum_{i=1}^{a}\left(\left|N^{(i-1)}\triangle N^{(i)}\right|+1\right)+C_\sigma(\ell_a,[g_{a+1}/g,\infty)), &\text{ if }g>0\\
        0, &\text{ if }g=0.
    \end{cases}
\end{equation*}

\begin{proposition}\label{prop:ideal-potential}
    Let $g\in[0,g_r]$. Then, $\Phi^{\mathrm{id}}(g)$ is bounded from above by $6\sum_{i\in [n]}\slc_{\eta}(\mx_i,g_r)$.
    Further, $\Phi^{\mathrm{id}}(g)\ge0$, and $\Phi^{\mathrm{id}}(g)=0$ if and only if $g=0$.
    Finally, $\Phi^{\mathrm{id}}(g)$ is monotonically non-increasing as $g$ decreases, and $\Phi^{\mathrm{id}}(g)<\Phi^{\mathrm{id}}(g')$ whenever $g<g'$ and there exists $j\in[r]$ with $g_j\in(g,g')$.
\end{proposition}
\begin{proof}
By \Cref{thm:polarized-approx}, part~\ref{item-bound}, $\Phi^{\mathrm{id}}(g)$ is bounded from above by 
\begin{equation*}
    1+2\sum_{i\in [n]}\slc_{\eta}(\mx_i,g_r)+r+C_\sigma(\ell_a,[0,\infty))\le 6\sum_{i\in [n]}\slc_{\eta}(\mx_i,g_r)
\end{equation*}
for all $g\in[0,g_r]$. 
By definition, $\Phi^{\mathrm{id}}(0)=0$. 
Further, all components of the function are non-negative.
If $g>0$, then $\Phi^{\mathrm{id}}(g)$ is at least $1$.

It follows directly from the definition of $\Phi^{\mathrm{id}}(g)$ that it is monotonically non-increasing as $g$ decreases within a fixed polarized interval.
Therefore, it remains to show for all $g\in(0,g_r]$ that $\Phi^{\mathrm{id}}(g_a)<\Phi^{\mathrm{id}}(g)$, where $a := \max\{i\in \Z:g_i<g\}$.
Since $\Phi^{\mathrm{id}}(g)=0$ if and only if $g=0$, we may assume that $a\ge1$.

On a high level, observe that when the gap drops to $g_a$, the partition in the definition of the potential changes from $N^{(a)}$ to $N^{(a-1)}$. Thus, by \Cref{lemma:partition-change}, the second term of $\Phi^{\mathrm{id}}(g)$ may increase by $|N^{(a-1)}\triangle N^{(a)}|$, but this is made up by the first term decreasing by $|N^{(a-1)}\triangle N^{(a)}|+1$.
To be more precise, we compute 
\begin{equation}\label{eq:potential-change}
    \Phi^{\mathrm{id}}(g)-\Phi^{\mathrm{id}}(g_a)=|N^{(a-1)}\triangle N^{(a)}|+1+C_{\sigma}(\ell_a,[g_{a+1}/g,\infty)])-C_\sigma(\ell_{a-1},[1,\infty)]).
\end{equation}
By \Cref{lemma:singular-value-scaling} and \Cref{thm:polarized-approx}, part~\ref{item-B} and~\ref{item-N}, we have that 
$\sigma(\ell_a)=\sigma(\ell_{N^{(a+1)}}^{\mathrm{diag}(h(g))^{-1}W})\cdot g_{a+1}/g$. 
Thus, we get that
\begin{align*}
C_\sigma(\ell_a,[g_{a+1}/g,\infty))&\ge C_\sigma(\ell_a,[g_{a+1}/g_a,\infty))\\
&=C_\sigma(\ell^{\mathrm{diag}(h(g_a))^{-1}W}_{N^{(a+1)}},[1,\infty)])\\
&\ge C_\sigma(\ell_{N^{(a)}}^{\mathrm{diag}(h(g_a))^{-1}W},[1,\infty))-|N^{(a-1)}\triangle N^{(a)}|\tag{by \Cref{lemma:partition-change}}\\
&=C_\sigma(\ell_{a-1},[1,\infty))-|N^{(a-1)}\triangle N^{(a)}|
\end{align*}
Inserting this into~\eqref{eq:potential-change} yields the desired result.
\end{proof}

\begin{theorem}\label{thm:amortized}
Algorithm~\ref{alg:implementation} terminates in 
\[O\left(n^2\log n \sum_{i\in[n]}\slc_{1/2}(x_i^\mcp,g(x^{(0)}))\right)\]
\textsc{Ratio-Circuit} augmentation steps.
\end{theorem}
\begin{proof}
    In the context of this proof, an \emph{iteration} of Algorithm~\ref{alg:implementation} refers to an iteration of the main while loop (Line~\ref{alg2:main-while}).
    We begin by observing that every iteration performs $O(n^2\log(n))$ \textsc{Ratio-Circuit} augmentation steps.
    Further, any augmentation step $z^*$ is computed as a solution to \textsc{Ratio-Circuit}$(A,c,\zero,1/x)$ or to~\eqref{sys:long_steps_lp}, which satisfies $\langle c,z^*\rangle\le0$. 
    Therefore, the optimality gap of the iterates are monotonically nonincreasing.

    Let $x\in P$ be the iterate at the start of some iteration, and assume that $x$ is not optimal.
    Let $x'\in P$ be the iterate two iterations later (we set $x'$ as the returned optimal solution if there is only one iteration remaining).
    It suffices to prove that 
    \begin{equation}
        \Phi^{\mathrm{id}}(g(x'))<\Phi^{\mathrm{id}}(g(x)).\label{potential-decrease}
    \end{equation}
    This is because by \Cref{prop:ideal-potential},  $\Phi^{\mathrm{id}}(g(x)) = O(\sum_{i\in[n]}\slc_{1/2}(x_i^\mcp,g(x^{(0)})))$, and $\Phi^{\mathrm{id}}(g(x)) = 0$ implies that $x$ is optimal.

    Let us apply \Cref{thm:polarized-approx} with $\eta=1/2$.
    We obtain a polarized decomposition $g_0,\dots,g_r$ with associated polarized partitions $(B^{(j)},N^{(j)})$ and polarized approximation $h_i$ of $x^\mcp_i$ for all $i\in [n]$.
    Let $g:=g(x)$ be the gap of $x$, and let $[g_a,g_{a+1}]$ be the polarized interval containing $g$.
    We define $W_a:=\mathrm{diag}(h(g_{a+1}))^{-1}W$ and $\ell^{\mathrm{id}}_a:=\ell_{N^{(a)}}^{W_a}$.
    As in \Cref{sec:bound}, we let $k^{\mathrm{id}}:=C_\sigma(\ell^{\mathrm{id}}_a,[g_{a+1}/g,\infty))$ and set $\hat g^\mathrm{id}:=\max\{g_a,g_{a+1}/\sigma_{k^\mathrm{id}}(\ell_a^\mathrm{id})\}$ if $k^{\mathrm{id}}>0$, and $\hat g^\mathrm{id}:=g_a$ otherwise.
    Note that by \Cref{prop:ideal-potential}, $\hat g^\mathrm{id}$ is the threshold gap at which $\Phi^{\mathrm{id}}$ decreases by at least $1$.
    
    If $g/\hat g^\mathrm{id}\le 4f_1(n)$, then~\eqref{potential-decrease} holds by \Cref{iterated-geometric-progress}, since the algorithm performs at least $n\log(4f_1(n))$ augmentation steps given by \textsc{Ratio-Circuit}$(A,c,\zero,1/x)$ (Line~\ref{alg2:slow-for}).
    Otherwise, $g/\hat g^\mathrm{id}> 4f_1(n)$.
    We note that by \Cref{lem:gap_approx}, $\lambda$ is an $n$-approximation of $g(x)$.
    We approximate the max central path twice, and perform at least $n\log(4f_1(n))$ augmentation steps given by \textsc{Ratio-Circuit}$(A,c,\zero,1/x)$ (Line~\ref{alg2:slow-for}) in between.    
    If the gap of the resulting iterate is not in $[g_a,g_{a+1}]$, then the potential function has decreased. 
    Thus, we can assume from now on that the gap is contained in $[g_a,g_{a+1}]$ and it follows from \Cref{guess_partition} that $(\tilde B,\tilde N)=(B^{(a)},N^{(a)})$.

    Consider the lifting operator $\ell_{\tilde N}^{W'}$, where $W' := \mathrm{diag}(\mx(g))^{-1}W$.
    Let $\sigma := \sigma(\ell^{W'}_{\tilde N})$ be its singular values.
    From \Cref{thm:polarized-approx}~\ref{item-approx}, we know that $h(g)\leq \mx(g) \leq 4h(g)$.
    Hence, applying \Cref{lemma:singular-value-scaling} yields
    \begin{equation}\label{singular-values-algorithm1}
        \frac{\sigma(\ell^{\mathrm{id}}_a)}{g_{a+1}}=\frac{\sigma(\ell^{\mathrm{diag(h(g))^{-1}W}}_{N^{(a)}})}{g} \leq \left\|\frac{h_B(g)}{\mx_B(g)}\right\|_\infty \left\|\frac{\mx_N(g)}{h_N(g)}\right\|_\infty \frac\sigma g  \leq \frac{4\sigma}{g},
    \end{equation}
    where the first equality follows from \Cref{thm:polarized-approx}, part~\ref{item-B} and~\ref{item-N}.
    By using the other inequality in \Cref{lemma:singular-value-scaling}, we obtain that 
    \begin{equation}\label{singular-values-algorithm2}
    \frac{\sigma(\ell^{\mathrm{id}}_a)}{g_{a+1}}\geq \frac{\sigma}{4g}.
    \end{equation}

    In the next step, we argue that $k^{\mathrm{id}}=k$ in the case that the potential function has not decreased, and the current iterate still requires more than $n\log(4f_1(n))$ augmentation steps computed by \textsc{Ratio-Circuit}$(A,c,\zero,1/x)$ to decrease the potential function.
    Recall that 
    \[k^{\mathrm{id}}=C_\sigma(\ell^{\mathrm{id}}_a,[g_{a+1}/g,\infty))=\max\{i:\sigma_i(\ell_a^{\mathrm{id}})\ge\frac{g_{a+1}}{g}\}=\max\{i:\sigma_i(\ell_{N^{(a)}}^{\mathrm{diag}(h(g))^{-1}W})\ge1\},\]
    where the second equality follows from \Cref{thm:polarized-approx}, part~\ref{item-B} and~\ref{item-N}.
    In contrast, the algorithm computes 
    \[k:=\max\{i:\tilde\sigma_i>\frac{\lambda^{(0)}}{\lambda^{(1)}}4n^2\}.\]
    Observe that by \Cref{iterated-geometric-progress} and \Cref{lem:gap_approx}, $\lambda^{(0)}/\lambda^{(1)}\ge 4f_1(n)/n^2$.
    Since $\sigma(\ell_{N^{(a)}}^{\mathrm{diag}(h(g))^{-1}W})$ and $\tilde\sigma$ satisfy 
    \[\frac{1}{16n}\cdot\tilde\sigma\le\sigma(\ell_{N^{(a)}}^{\mathrm{diag}(h(g))^{-1}W})\le16n\cdot\tilde\sigma\]
    by~\eqref{eq:singular_values},~\eqref{singular-values-algorithm1} and~\eqref{singular-values-algorithm2}, we obtain that $k\le k^{\mathrm{id}}$ in case the potential function has not decreased.
    A symmetric argument shows that if the next $n\log(4f_1(n))$ augmentation steps computed by \textsc{Ratio-Circuit}$(A,c,\zero,1/x)$ do not suffice to decrease the potential function, then $k\ge k^{\mathrm{id}}$.
    Thus, we can assume without loss of generality that $k=k^{\mathrm{id}}$.
    Therefore, if $k=\dim(\pi_{\tilde N}(W))$, then also $k^{\mathrm{id}}=\dim(\pi_{\tilde N}(W))$  and~\eqref{potential-decrease} holds by \Cref{iterated-geometric-progress} and \Cref{smallest_singular_value}.

    Let $\hat g:= \max\{g_a, g_{a+1}/\sigma_{k}\}$.
    By~\eqref{singular-values-algorithm2}, we have $\hat g^{\mathrm{id}}\geq \hat g/4$.
    We compute $\bar x$ and $y^{(0)}$, satisfying Properties \ref{cond:x_gap}, \ref{cond:x_lbound}, \ref{cond:y_gap}, \ref{cond:y_lbound}.
    Recall that we can assume that the gap of the current iterate is still in $[g_a, g_{a+1}]$.  
    Thus, by \Cref{scaling-up}, the iterate achieves a gap of at most $f_1(n)\cdot\hat g$ after running Algorithm~\ref{alg:long-steps-forced}.
    Further, by \Cref{thm:polarized-approx}~\ref{item-approx} and \Cref{lemma:singular-value-scaling}, $\hat g/4\le\hat g^\mathrm{id}$. 
    Therefore, $n\log(4f_1(n))$ augmentation steps computed by \textsc{Ratio-Circuit}$(A,c,\zero,1/x)$ suffice in order to reduce the gap of the iterate to at most $\hat g^\mathrm{id}$, which will happen in the next iteration of Line~\ref{alg2:slow-for}.
    This proves the last case of~\eqref{potential-decrease} and concludes the proof of the theorem.    
\end{proof}

\section{Acknowledgments}

Part of this work was completed while the first and third author were participating in the 
Oberwolfach workshop on "Combinatorial Optimization", November 10-14, 2024. 
Part of this work was completed while the second author visited CWI in Amsterdam, supported by the Fonds de la Recherche Scientifique - FNRS.
Daniel Dadush and Zhuan Khye Koh were supported by European Research Council grant 805241-QIP. 
Stefan Kober was supported by Fonds de la Recherche Scientifique - FNRS through research project BD-DELTA-3 (PDR 40028812).

\bibliographystyle{alpha}
\bibliography{references}

\newpage 

\appendix

\section{Omitted Proofs}

\ratiocircuitelementary*

\begin{proof}
We start by proving the following simple proposition.

\begin{proposition}\label{pro:circuit_lp}
Given a matrix $A\in \R^{m\times n}$, a vector $c\in \R^n$ and a scalar $\alpha\in \R$, consider the following polyhedron
\[P := \left\{x\in \R^n: Ax = \zero,\; \pr{c}{x}\leq \alpha,\; x\geq \zero\right\}.\]
Every extreme point of $P$ is either $\zero$ or an elementary vector of $A$.
\end{proposition}

\begin{proof}
Let $x\neq \zero$ be an extreme point of $P$.
Clearly, $x\in \ker(A)$.
For the purpose of contradiction, suppose that $x\notin \elementary{A}$.
By conformal decomposability, we can write $x = p + q$ for two nonnegative and linearly independent vectors $p,q\in \ker(A)\setminus \{\zero\}$.
First, we claim that $\pr{c}{p},\pr{c}{q}\neq 0$.
Suppose for a contradiction that $\pr{c}{p} = 0$ say.
Let $y := 0.5p + q$ and $z := 1.5p + q$; note that $y$ and $z$ are distinct points in $P$.
Since $x = 0.5y + 0.5z$, it follows that $x$ is not an extreme point; a contradiction.

Next, let 
\[y' := (1+\varepsilon)p + \left(1-\frac{\pr{c}{p}}{\pr{c}{q}}\varepsilon\right)q \qquad \qquad z' := (1-\varepsilon)p + \left(1+\frac{\pr{c}{p}}{\pr{c}{q}}\varepsilon\right)q\]
for a sufficiently small $\varepsilon>0$ such that $y',z'\in P$.
It is easy to check that $y'\neq z'$.
Since $x = y' + z'$, we arrive at the same contradiction again.
\end{proof}

Now, let $(z^+,z^-)\neq \zero$ be a basic optimal solution to \eqref{sys:min_ratio_lp}.
From \Cref{pro:circuit_lp}, we know that $(z^+,z^-)\in \elementary{[A, -A]}$.
If $z^+_i, z^-_i>0$ for some $i\in [n]$, then it follows from support minimality that $z^+_i = z^-_i$ and $z^+_j = z^-_j = 0$ for all $j\in [n]\setminus i$.
In this case, the corresponding solution $z$ to \eqref{sys:min_ratio} is $\zero$.
Otherwise, we have $z^+_i = 0$ or $z^-_i = 0$ for all $i\in [n]$.
So, the corresponding solution $z$ to $\eqref{sys:min_ratio}$ is an elementary vector of $A$.
\end{proof}

\projections*

\begin{proof}
From the definition of $z^p$, we have $z^p_N = \Pi_{V_d}(z_N)$.
Let $z^\perp_N:= \Pi_{V^\perp_d}(z_N)$, note that $z^p_N + z^\perp_N = z_N$ is an orthogonal decomposition.
It is left to show that $\|z^\perp_N\|\leq \|z_B\|/\sigma_k$.

Since $z^p_N\in V_d$ and $z^\perp_N\in V^\perp_d$, $\ell^W_N(z^p_N)$ and $\ell^W_N(z^\perp_N)$ are generated by two sets $U_1$ and $U_2$ of left singular vectors respectively such that $\pr{u_1}{u_2} = 0$ for all $u_1\in U_1$ and $u_2\in U_2$. 
It follows that $\ell^W_N(z^p_N)$ and $\ell^W_N(z^\perp_N)$ are orthogonal.
Hence,
\begin{equation*}
 \|z_B\| \geq \|\ell_N^W(z_N)\| = \|\ell_N^W(z_N^p+z_N^\perp)\| = \|\ell_N^W(z_N^p) + \ell^W_N(z_N^\perp)\| \geq \|\ell_N^W(z_N^\perp)\| \geq \sigma_k \|z_N^\perp\|, 
\end{equation*}
where the last inequality is due to $\sigma_{\min}(\ell^W_N|{_{V^\perp_d}})\geq \sigma_k$.
Indeed, $V^\perp_d$ is spanned by right singular vectors associated with $\sigma_1, \sigma_2, \dots, \sigma_k$ because $V_d$ is a singular subspace of dimension $d$.
\end{proof}

\end{document}